\documentclass[12pt]{amsart}
\usepackage{amssymb}
\usepackage{amsmath, graphicx, rotating}
\usepackage{color}
\usepackage{soul}
\usepackage{todonotes}

\usepackage{dsfont}
\usepackage{mathrsfs}

\usepackage[T1]{fontenc}
\usepackage{lmodern}

\usepackage[english]{babel}

\usepackage{ upgreek }
\usepackage{amsmath,amssymb}
\usepackage{amsthm}
\usepackage{graphicx}
\usepackage{float}
\usepackage{enumitem}

\usepackage{ bbm }
\usepackage{ stmaryrd }
\usepackage{ mathrsfs }
\usepackage{frcursive}

\usepackage{pgf, tikz}

\definecolor{rouge}{rgb}{0.7,0.00,0.00}
\definecolor{vert}{rgb}{0.00,0.5,0.00}
\definecolor{bleu}{rgb}{0.00,0.00,0.8}

\newtheorem{theorem}{Theorem}[section]
\newtheorem{lemma}[theorem]{Lemma}

\newtheorem{corollary}[theorem]{Corollary}
\newtheorem{remark}[theorem]{Remark}
\newtheorem{proposition}[theorem]{Proposition}

\numberwithin{equation}{section}

\def\geq{\geqslant}
\def\leq{\leqslant}

\def\p{\mathbb{P}}
\def\e{\mathbb{E}}

\def\({\left(}
\def\){\right)}

\begin{document}
\title[Harmonic moments of $Z_n$ for a BPRE]{Harmonic moments and large deviations for a supercritical branching process in a random environment}

\author{Ion~Grama}
\curraddr[Grama, I.]{ Universit\'{e} de Bretagne-Sud, LMBA UMR CNRS 6205,
Vannes, France}
\email{ion.grama@univ-ubs.fr}

\author{Quansheng Liu}
\curraddr[Liu, Q.]{ Universit\'{e} de Bretagne-Sud, LMBA UMR CNRS 6205,
Vannes, France}
\email{quansheng.liu@univ-ubs.fr}

\author{Eric Miqueu}
\curraddr[Miqueu, E.]{Universit\'{e} de Bretagne-Sud, LMBA UMR CNRS 6205,
Vannes, France}
\email{eric.miqueu@univ-ubs.fr}
\date{\today }

\begin{abstract}
Let $(Z_n)$ be a supercritical branching process in an independent and identically distributed random environment $\xi$.
We study the asymptotic of the harmonic moments $\mathbb{E}\left[Z_n^{-r} | Z_0=k \right]$ of order $r>0$ as $n \to \infty$. We  exhibit a phase transition with the critical value $r_k>0$
determined by the equation
$\mathbb E p_1^k = \mathbb E m_0^{-r_k},$ where  $m_0=\sum_{k=0}^\infty k p_k$ with $p_k=\mathbb P(Z_1=k | \xi),$ 
assuming that $p_0=0.$ 
Contrary to the constant environment case (the Galton-Watson case), this critical value is different from that for the existence of the harmonic moments of $W=\lim_{n\to\infty} Z_n / \mathbb E (Z_n|\xi).$ The aforementioned phase transition  is linked to that for the rate function of the lower large deviation for $Z_n$.
As an application, we obtain a lower large deviation result for $Z_n$ under weaker conditions than in previous works and give a new expression of the rate function.
We also improve an earlier result about the convergence rate in the central limit theorem for $W-W_n,$ and find an equivalence for the large deviation probabilities of the ratio $Z_{n+1} / Z_n$.
\end{abstract}

\maketitle

\section{Introduction}

A branching process in a random environment (BPRE) is a natural and important generalisation
of the Galton-Watson process, where the reproduction law varies according to a random environment indexed by time.
It was introduced for the first time in Smith and Wilkinson \cite%
{smith} to modelize the growth of a population submitted to an environment.
For background concepts and basic results concerning a BPRE we refer to 
Athreya and Karlin \cite{athreya1971branching, athreya1971branching2}.
In the critical and subcritical regime the branching process goes out and the research interest has been mostly concentrated  on the survival probability
and conditional limit theorems,
see e.g. Afanasyev, B\"oinghoff, Kersting, Vatutin \cite{afanasyev2012limit, afanasyev2014conditional}, Vatutin \cite{Va2010},  Vatutin and Zheng \cite{VaZheng2012}, and the references therein.
In the supercritical
case, a great deal of current research has been focused on large deviations, 
see e.g. Bansaye and Berestycki \cite{bansaye2009large}, Bansaye and B\"oinghoff \cite{bansaye2011upper, bansaye2013lower, bansaye2014small}, B\"oinghoff and Kersting \cite{boinghoff2010upper}, Huang and Liu \cite{liu} and Nakashima \cite{Nakashima2013lower}. 
In the particular case when the offspring distribution has a fractional linear generating function, precise asymptotics can be found in 
B\"oinghoff \cite{boinghoff2014limit} and Kozlov \cite{kozlov2006large}.
An important closely linked issue is the asymptotic behavior of the harmonic moments $\mathbb{E} [Z_n^{-r} | Z_0 = k]$ of the process $Z_n$ starting with $Z_0=k$ initial individuals. For the Galton-Watson process which corresponds to the constant environment case, the question has been studied exhaustively in Ney and Vidyashankar \cite{ney2003harmonic}.
For a BPRE, it has only been partially treated  in \cite[Theorem 1.3]{liu}.

In the present paper, we give a complete description of the asymptotic behavior of the harmonic moments $\mathbb{E}_k [Z_n^{-r} ] = \mathbb{E} [Z_n^{-r} | Z_0 = k]$ of the process $Z_n$ starting with $k$ individuals and assuming that each individual gives birth to at least one offspring (non-extinction case). 
As a consequence, we improve the lower large deviation result for the process $Z_n$ obtained in \cite[Theorem 3.1]{bansaye2013lower} by relaxing the hypothesis therein. 
In the meanwhile we give a new characterization of the rate function in the large deviation result stated in \cite{bansaye2013lower}. 
We also improve the exponential convergence rate in the central limit theorem for $W-W_n$ established in \cite{liu}.
Furthermore, we investigate 
the large deviation behavior of the ratio $R_n = \frac{Z_{n+1}}{Z_n}$, i.e. the asymptotic of the large deviation probability $ \p ( | R_n - m_n|  > a ) $ for $a>0$, where $m_n$ is the expected value of the number of children of an individual in generation $n$ given the environment $\xi$. For the Galton-Watson process, the quantity $R_n$ is the Lotka-Nagaev estimator of the mean $\e Z_1$, whose large deviation probability has been studied in \cite{ney2003harmonic}.  

Let us explain briefly the findings of the paper in the special case when we start with $Z_0=1$ individual. Assume that $\p (Z_1=0)=0$
and $\p (Z_1=1)>0$. Define $r_1$ as the solution of the equation 
\begin{equation}
\label{eq rk}
\e m_0^{-r_1} = \gamma,
\end{equation}
with
\begin{equation}
\gamma =\p (Z_1=1). 
\end{equation} 
From Theorem \ref{thm harmonic moments Zn} we get the following asymptotic behavior of the harmonic moments $\e \left[ Z_n^{-r} \right]$ for $r>0$.  Assume that $\e m_0^{r_1+ \varepsilon} < \infty$ for some $\varepsilon>0$. Then, we have
\begin{equation}
\label{equivalent Z_n intro}
\left\{ \begin{array}{l c l l}
\displaystyle{\frac{ \mathbb{E} \left[ Z_n^{-r} \right] }{\gamma^n}} & \underset{n \to \infty}{\longrightarrow} & 
C(r) & \quad \text{if} \ \ r > r_1 , \\
& & &\\
\displaystyle{\frac{ \mathbb{E} \left[ Z_n^{-r} \right]}{n \gamma^n}} &\underset{n \to \infty}{\longrightarrow}& 
C(r)   & \quad \text{if} \ \ r = r_1, \\
& & &\\
\displaystyle{\frac{ \mathbb{E} \left[ Z_n^{-r} \right]}{\left(\e m_0^{-r} \right)^n}} &\underset{n \to \infty}{\longrightarrow}& 
C(r)  & \quad \text{if} \ \ r < r_1, \\
\end{array}
\right.
\end{equation}
where $C(r)$ are positive constants for which we find integral expressions.
This shows that there are three phase transitions in the rate of convergence of the harmonic moments for the process $Z_n$, with the critical value $r_1$.
It generalizes the result of  \cite{ney2003harmonic} for the Galton-Watson process. 
For a BPRE, it completes and improves the result of \cite{liu}, where the asymptotic equivalent of the quantity $\mathbb{E} \left[ Z_n^{-r} \right]$ 
has been established in the particular case where $r<r_1$ and under stronger assumptions.

The proof presented here is new and straightforward compared to that for the Galton-Watson process given in \cite{ney2003harmonic}. Indeed, we prove \eqref{equivalent Z_n intro} starting from the branching property 
\begin{equation}
\label{decomposition Zn1}
Z_{n+m} = \sum_{i=1}^{Z_m} Z_{n,i}^{(m)},
\end{equation}
where conditionally on the environment $\xi$, for $i \geq 1$, the sequences of random variables $\{ Z_{n,i}^{(m)} : n \geq 0 \}$ are i.i.d. branching processes with the shifted environment $T^m (\xi_0, \xi_1 , \ldots ) = (\xi_m, \xi_{m+1} , \ldots )$, and are also independent of $Z_m$.
This simple idea leads to the following equation which will play a key role in our arguments:
\begin{equation}
\label{eq inequation Zn-r}
\mathbb{E} \left[ Z_{n+1}^{-r} \right] = \gamma^{n+1}  + \sum_{j=0}^{n} b_j \gamma^{n-j}  c_r^j,
\end{equation}
where $c_r= \e m_0^{-r}$ and $(b_j)_{j \geq 0}$ is an increasing and bounded sequence.
Such a relation highlights the main role played by the quantities $\gamma$ and $c_r$ in the asymptotic study of $\e [Z_n^{-r}]$ whose behavior depends on whether $\gamma < c_r$, $\gamma = c_r$ or $\gamma > c_r$.
Note that the complete proof of \eqref{equivalent Z_n intro} relies on some recent and important results established in \cite{glm2016asymptotic} concerning the critical value for the existence of the harmonic moments of the r.v. $W$ and the asymptotic behavior of the distribution $\p (Z_n = j)$ as $n \to \infty$, with $ j \geq 1$.
For the Galton-Watson process, our approach based on \eqref{eq inequation Zn-r} is much simpler than that in \cite{ney2003harmonic}.

Our proof also gives an expression of the limit constants in \eqref{equivalent Z_n intro}. For the Galton-Watson process,   it recovers the expressions of \cite[Theorem 1]{ney2003harmonic} in the cases where $\gamma > c_r$ and $ \gamma < c_r$. In the critical case where $r=r_1$, the limit constant obtained in this paper is different to that of \cite[Theorem 1]{ney2003harmonic}, which leads to an alternative expression of the constant and the following nice identity involving the well-known functions $G$, $Q$ and $\phi$:  defining
\begin{eqnarray*}
G(t) &=& \sum_{k=0}^{\infty} t^k \p (Z_1=k), \\
Q (t) &=& \lim_{n \to \infty} \gamma^{-n} G^{\circ n} (t) , \\
\phi (t) &=& \e [ e^{-t W} ],
\end{eqnarray*}
and  denoting $m=\e [Z_1]$, $\gamma= \p (Z_1=1)$ and  $\bar G (t) = G(t) - \gamma t$, we have
\begin{equation}
\label{eq egalite intro}
\frac{1}{\gamma}  \int_{0}^{\infty} \bar{G}( \phi (u) )  u^{r-1} du = \int_{1}^{m} Q ( \phi (u) )  u^{r-1} du.
\end{equation}
For a BPRE, we will show a generalization of \eqref{eq egalite intro} in Proposition \ref{prop relation G Q psi BPRE}.

As a consequence of Theorem \ref{thm harmonic moments Zn} and of a version of the G\"artner-Ellis theorem, we obtain a lower large deviation result for $Z_n$ under conditions weaker than those in \cite[Theorem 3.1]{bansaye2013lower}.
Assume that $\p (Z_1=0)=0$ and $\e m_0^{r_1+ \varepsilon} < \infty$ for some $\varepsilon>0$.
Let 
\begin{equation}
\label{Lambda}
\Lambda ( \lambda )= \log \e e^{\lambda X}
\end{equation}
 be the log-Laplace transform of $X= \log m_0$ and
\begin{equation}
\label{Lambda*}
\Lambda^* (x)= \sup_{\lambda \leq 0 } \{\lambda x - \Lambda ( \lambda ) \}
\end{equation} 
be the Fenchel-Legendre transform of $\Lambda (\cdot)$.
Then, for any 
$\theta \in (0, \e [X])$, we have 
\begin{eqnarray}
\label{eq PGD GLM intro}
 \lim_{n \to \infty} -\frac{1}{n} \log \p \left(  Z_n \leq e^{\theta n }\right)  = \chi^* (\theta) \in (0, \infty),
\end{eqnarray}
where 
\begin{eqnarray}
\label{chi intro}
\chi^* (\theta)  
&=&
\left\{ 
\begin{array}{c l}
-r_1 \theta - \log \gamma & \text{if} \ 0 < \theta < \theta_1, \\ 
\Lambda^* (\theta) &  \text{if} \ \theta_1 \leq \theta < \e [X],
\end{array}
\right.
\end{eqnarray}
with
\begin{equation}
\label{theta_k intro}
\theta_1 = \Lambda ' (-r_1) \in (0, \e [X]).
\end{equation}
Equation \eqref{eq PGD GLM intro} improves the result of \cite[Theorem 3.1(ii)]{bansaye2013lower} in the case when $\p (Z_1=0)=0$, since it is assumed in \cite{bansaye2013lower} that $\e m_0^t< \infty$ for all $t>0$, whereas we only require that $\e m_0^{r_1+ \varepsilon} < \infty$ for some $\varepsilon>0$. Moreover, equations \eqref{chi intro} and \eqref{theta_k intro}  also give new and alternative expressions of the rate function and the critical value. In fact, it has been proved in \cite{bansaye2013lower} that, in the case when $\p (Z_1=0)=0$ and $Z_0=1$,
\begin{equation}
\lim_{n \to \infty} -\frac{1}{n} \log \p \left(  Z_n \leq e^{\theta n }\right)  = I (\theta) \in (0, \infty),
\end{equation}
with
\begin{eqnarray}
\label{eq chi thm000}
I (\theta)  
&=&
\left\{ 
\begin{array}{c l}
\rho \left(1 - \frac{\theta}{\theta_1^*} \right) + \frac{\theta}{\theta_1^*}
\Lambda^* (\theta_1^*) & \text{if} \ 0 < \theta < \theta_1^*, \\ 
\Lambda^* (\theta) &  \text{if} \ \theta_1^* \leq \theta < \e [X],
\end{array}
\right.
\end{eqnarray}
where $\rho = - \log \gamma $ and $\theta_1^* $ the unique solution on $(0, \e [X])$ of the equation 
\begin{equation}
\label{theta_* bansaye intro}
\frac{\rho- \Lambda^* (\theta_1^*)}{\theta_1^*} = \inf_{0 \leq \theta \leq \e [X]} \frac{\rho - \Lambda^* (\theta)}{\theta}.
\end{equation}
It follows directly from  the relations \eqref {chi intro} to \eqref{theta_* bansaye intro} that $\theta_1 = \theta_1^*$ and $\chi^* (\theta) = I (\theta)$ for all $\theta \in (0, \e [X])$.This fact can also be shown  by using simple duality arguments between $\Lambda$ and $\Lambda^*$, as will be seen in the next section.

The rest of the paper is organized as follows. In 
Section \ref{sec main results} we give the precise statements of the main theorems with applications. 
Section \ref{sec preuve main thm} is devoted to the proof of the main results, Theorems \ref{thm harmonic moments Zn} and \ref{thm PGD}. The proofs for the applications  are deferred to Section \ref{sec application}.

Throughout the paper, we denote by $C$ an absolute constant whose value may differ from line to line. 

\section{Main results}
\label{sec main results}

A BPRE $(Z_n)$ can be described as follows.
The random environment is
represented  by a sequence $\xi = (\xi_0, \xi_1 , ... ) $ of independent and
identically distributed random variables (i.i.d.\ r.v.'s) taking values in an abstract space $\Xi$, whose
realizations determine the probability generating functions
\begin{equation}
\mathnormal{f}_{\xi_n} (s)= f_n (s) = \sum_{i=0}^{\infty} p_i ( \xi_n ) s^i,
\quad s \in [0,1], \quad p_i ( \xi_n ) \geq 0, \quad \sum_{i=0}^{ \infty}
p_i (\xi_n) =1.  \label{defin001}
\end{equation}
The BPRE $(Z_n)_{n \geq 0}$ is defined by the relations
\begin{equation}  \label{relation recurrence Zn}
Z_0 = 1, \quad Z_{n+1} = \sum_{i=1}^{Z_n} N_{n, i}, \quad \text{for} \quad n
\geq 0,
\end{equation}
where the random variables $N_{n,i} $ $(i = 1, 2, \dots)$ represent the number of children of the $i$-th individual of the
generation $n$.
Conditionally on the environment $\xi $, the r.v.'s $N_{n,i} $
$(n \geq 0, i \geq 1)$ are independent of each other, and each  $N_{n,i} $
$( i \geq 1)$ has common probability generating function $\mathnormal{f}_n.$

In the sequel we denote by $\mathbb{P}_{\xi}$ the
\textit{quenched law}, i.e.\ the conditional probability when the
environment $\xi$ is given, and by $\tau $ the law of the environment $\xi$.
Then
$\mathbb{P}(dx,d\xi) = \mathbb{P}_{\xi}(dx) {\tau}(d\xi)$
is the total law of the process, called
\textit{annealed law}. The corresponding quenched and annealed expectations
are denoted respectively by $\mathbb{E}_{\xi}$ and $\mathbb{E}$.  We also denote by $\mathbb{P}_k$ and $\mathbb{E}_k$ the corresponding annealed probability and expectation starting with $Z_0=k$ individuals, with $\mathbb P_1 = \mathbb P$ and $\mathbb E_1 = \mathbb E.$   
From \eqref{relation recurrence Zn}, it follows that the  probability generating function of $Z_n$ conditionally on the environment $\xi$ is
given by
\begin{equation}
\label{gn}
g_n (t) = \e_{\xi} [ t^{Z_n} ]  =  f_0 \circ \ldots \circ \ f_{n-1} (t), \quad 0 \leq t \leq 1. 
\end{equation} 
Since $\xi_0, \xi_1 , \ldots$ are i.i.d. r.v.'s, we get that the annealed probability generating function $G_{k,n}$ of $Z_n$ starting with $Z_0=k$ individuals is given by
\begin{equation}
\label{Gkn}
G_{k,n} (t) = \mathbb{E}_k [t^{Z_n}] = \mathbb{E} \left[ g_n^k (t) \right], \quad 0 \leq t \leq 1. 
\end{equation}
We also
define, for $n\geq 0$,
\begin{equation*}
m_n = m ( \xi_n )= \sum_{i=0}^\infty i p_i ( \xi_n ) \quad \text{and} \ \ \Pi_n = \mathbb{E}_{\xi} Z_n = m_0 ... m_{n-1},
\end{equation*}
where $m_n $ represents the average number of children of an
individual of generation $n$  when the environment $\xi $ is given, and $\Pi_0=1$ by convention. Let
\begin{equation} \label{Wn}
W_n =\frac{Z_n}{\Pi_n} , \quad n\geq 0,
\end{equation}
be the normalized population size.
It is well known that under the quenched law $%
\mathbb{P}_{\xi}$, as well as under the annealed law $\p$, the sequence $(W_n)_{n \geq 0} $ is a non-negative martingale with respect to the filtration
$$\mathcal{F}_n = \sigma
\left(\xi, N_{k,i} , 0 \leq k \leq n-1, i = 1,2 \ldots \right), $$
where by convention $\mathcal{F}_0 = \sigma(\xi)$.
Then the limit $W = \lim W_n $ exists $\mathbb{P}$ - a.s. and $\mathbb{E} W \leq 1 $.
We also denote the quenched and annealed Laplace transform of $W$  by
\begin{equation}
\label{quenched annealed laplace W}
\phi_{\xi} (t) =  \mathbb{E}_{\xi} \left[ e^{-t W} \right] \quad  \text{and} \quad \phi (t) = \mathbb{E} \left[ e^{-t W} \right], \quad \text{for}\ t \geq 0,
\end{equation}
while starting with $Z_0=1$ individual. For $k\geq 1$, while starting with $Z_0=k$ individuals, we have
\begin{equation}
\label{annealed laplace W Z_0=k}
\phi_{k} (t) := \mathbb{E}_{k} \left[ e^{-t W}  \right] = \e [ \phi_{\xi}^k (t) ].
\end{equation}
Another important tool in the study  of a BPRE  is the associated random
walk
\begin{equation*}
S_n = \log \Pi_n = \sum_{i=1}^{n} X_i , \quad n \geq 1, 
\end{equation*}
where the r.v.'s  $X_i = \log m_{i-1}$ $(i\geq1)$ are i.i.d.\  depending only on the environment $\xi$.
For the sake of brevity, we set $X=  \log m_0 $ and 
\begin{equation*}
\mu = \mathbb{E} X . 
\end{equation*}
We shall consider a supercritical BPRE where $\mu \in (0, \infty)$, so that under the extra condition $\mathbb E |\log (1-p_0(\xi_0))| <\infty$ (see \cite{smith}), 
the population size tends to infinity with positive probability.
For our propose, in fact we will assume in the whole paper that each individual gives birth to at least one child, i.e. 
\begin{equation}
\label{condition p0=0}
p_0(\xi_0) = 0 \quad a.s.
\end{equation}
Therefore, under the condition 
\begin{equation}
\label{CN CV L1 W}
\mathbb{E} \frac{Z_1}{m_0} \log^+ Z_1  < \infty ,
\end{equation}
the martingale $(W_n)$  converges to $W$ in $L^1 (\mathbb{P})$ (see e.g. \cite{tanny1988necessary}) and
\[ 
\mathbb{P} (W>0) =  \mathbb{P} (Z_n \to \infty)=1. 
\]
Now we can state the main result of the paper about the asymptotic of the harmonic moments $\e_k [Z_n^{-r}]$ of the process $(Z_n)$ for $r>0$, starting with $Z_0=k$ for $k \geq 1$.
Define
\begin{equation}
\label{eq gamma_k}
\gamma_k = \mathbb{P}_k (Z_1=k) = \mathbb{E} [ p_1^k (\xi_0)]
\end{equation}
and 
\begin{equation}
\bar{G}_{k,1} (t) = {G}_{k,1} (t) - \gamma_k t^k = \sum_{j=k+1}^{\infty } t^j \p (Z_1=j),
\end{equation}
where $G_{k,1}$ is the generating function of $Z_1$ defined in \eqref{Gkn}.
Let $r_k$ be the solution of the equation
\begin{equation}
\label{equation rj}
\gamma_k = \mathbb{E} \left[ m_0^{-r_k} \right],
\end{equation}
with the convention that $r_k= + \infty $ if $p_1 (\xi_0) = 0$ a.s.
For any $r>0$, set
\begin{equation}
c_r = \mathbb{E} m_0^{-r}.
\end{equation}
For any $k \geq 1$ and $r>0$, let $\mathbb{P}_k^{(r)}$ be the probability measure (depending on $r$) defined, for any $\mathcal{F}_n$-measurable r.v. $T$, by 
\begin{equation}
\label{changement de mesure}
\mathbb{E}_k^{(r)} [T] =\frac{\mathbb{E}_k \left[ \Pi_n^{-r} T \right]}{c_r^n}.
\end{equation}
Set $\mathbb{P}^{(r)}=\mathbb{P}_1^{(r)}$ and $\mathbb{E}^{(r)}=\mathbb{E}_1^{(r)}.$
It is easily seen that under $\p^{(r)}$, the process $(Z_n)$ is still a supercritical branching process in a random environment with $\p^{(r)}(Z_1=0)=0$, and $(W_n)$ is still a non-negative martingale which converges a.s. to $W$. Moreover, by \eqref{CN CV L1 W} and the fact that $m_0 \geq 1$, we have
\[
\e^{(r)} \left[ \frac{Z_1}{m_0} \log Z_1 \right] = \e \left[ \frac{Z_1}{m_0^{1+r}} \log Z_1\right] / \e \left[ m_0^{-r}\right] \leq \e \left[ \frac{Z_1}{m_0} \log Z_1 \right] / \e \left[ m_0^{-r}\right] < \infty,
\]
which implies that
\begin{equation}
\label{cv L1 Pr W}
W_n \rightarrow W \quad \text{in} \quad L^1 (\p^{(r)}).
\end{equation}
Then we have the following equivalent for the harmonic moments of $Z_n$. Let
\begin{equation}
\label{An}
A_{k,n} (r) = \left\{ \begin{matrix}
\gamma_k^{n}, & \ \text{if} \ \ r > r_k , \\
n \gamma_k^{n} , & \ \text{if} \ \ r = r_k, \\
c_r^{n} , & \ \text{if} \ \ r < r_k. \\
\end{matrix}
\right.
\end{equation}
\begin{theorem}
\label{thm harmonic moments Zn}
Let $k \geq 1$ and assume that $\e m_0^{r_k+ \varepsilon} < + \infty$ for some $\varepsilon>0$. Then we have 
\begin{equation}
\label{equivalent Z_n}
\lim_{n \to \infty}  \frac{\mathbb{E}_k \left[ Z_n^{-r} \right]}{A_{k,n} ( r)} = C(k, r) := \left\{ \begin{array}{l l}
\displaystyle{\frac{1}{\Gamma (r)} \int_0^{\infty} Q_k ( e^{-t} ) t^{r-1} dt} & \ \text{if}  \ r > r_k , \\
& \\
\displaystyle{\frac{\gamma_k^{-1} }{\Gamma (r)} \mathbb{E}^{(r)} \left[ \int_{0}^{\infty} \bar{G}_{k,1} ( \phi_{\xi} (t) )  t^{r-1} dt \right] } & \ \text{if}  \ r = r_k, \\
& \\
\displaystyle{ \frac{1}{\Gamma (r)} \int_0^{\infty} \phi_k^{(r)} (t) t^{r-1} dt } & \ \text{if} \  r < r_k, \\
\end{array}
\right.
\end{equation}
where $C(k,r) \in (0, \infty)$, $\Gamma(r)=\int_0^\infty x^{r-1}e^{-x}dx$ is the Gamma function, and $\phi_k^{(r)} (t)= \mathbb{E}_k^{(r)} [e^{-tW}]$ is the Laplace transform of $W$ under $\mathbb{P}_k^{(r)}$.
\end{theorem}
This theorem shows that there is a phase transition in the rate of convergence of the harmonic moments of $Z_n$ with the critical value $r_k>0$ defined by \eqref{equation rj}.
This critical value $r_k$ is generally different from the critical value $a_k$ for the existence of the harmonic 
moment of $W.$ 
Indeed, as shown in \cite[Theorem 2.1]{glm2016asymptotic} (see Lemma \ref{thm harmonic moments W} below),
the critical value $a_k$ is determined by
\begin{equation}
\label{equation ak}
\mathbb{E} \left[ p_1^k(\xi_0) m_0^{a_k} \right] = 1,
\end{equation}
which is in general different from the critical value $r_k$ determined by \eqref{equation rj}, that is 
\begin{equation}
\label{equation rk}
\mathbb{E} \left[ p_1^k (\xi_0) \right] = \mathbb{E} \left[ m_0^{-r_k} \right]. 
\end{equation}
This is in contrast with the Galton-Watson process where $r_k=a_k=r_1^k$, with $r_1$ the solution of the equation $p_1 m_0^{r_1}  = 1$ which coincides with both equations
(\ref{equation ak}) and (\ref{equation rk}).
Theorem \ref{thm harmonic moments Zn} generalizes the result of \cite{ney2003harmonic} for the Galton-Watson process. 
For a BPRE, it completes and improves Theorem 1.3 of \cite{liu}, where the formula (\ref{equivalent Z_n}) was only proved for $k=1$ and $r<r_1$, and under the following much stronger boundedness condition:
there exist some constants $p, c_0, c_1 >1$ such that
\[
c_0 \leq m_0 \leq c_1  \quad \text{and} \quad  c_0 \leq m_0 (p) \leq c_1 \quad a.s.,
\]
where $m_0(p) = \sum_{i=1}^\infty i^{p} p_i (\xi_0) $.
Instead of this boundedness condition, here we only require the moment assumption $\e [m_0^{r_k+ \varepsilon}] < \infty$ for some $\varepsilon>0$. 

For the Galton-Watson process with $k=1$ initial individuals, the expression of the limit constant in the case when $r=r_1$ (up to the constant factor $\Gamma (r)$) becomes 
\[
\frac{1}{\gamma}  \int_{0}^{\infty} \bar{G} ( \phi (u) )  u^{r-1} du,
\]
whereas it has been proved in \cite{ney2003harmonic} that the limit constant is equal to  
\[
\int_{1}^{m} Q ( \phi (u) )  u^{r-1} du.
\]
Actually, the above two expressions coincide, as shown by the following result  valid for a general BPRE. Let $Q_k (t)$ be defined by
\begin{equation}
\label{eq Qk}
Q_k (t) = \lim_{n \to \infty} \frac{G_{k,n}(t)}{\gamma_k^n} = \sum_{j=k}^{\infty} q_{k,j} t^j, \quad t \in [0,1),
\end{equation}
where $G_{k,n}$ is defined by \eqref{Gkn} and the limit exitsts according to \cite[Theorem 2.3]{glm2016asymptotic} (see Lemma \ref{thm small value probability 2} below).
\begin{proposition}
\label{prop relation G Q psi BPRE}
For $k \geq 1$ and $r=r_k$, we have
\begin{equation}
\label{eq egalite BPRE}
\frac{1}{\gamma_k}  \e^{(r)} \left[ \int_{0}^{\infty} \bar{G}_{k,1}( \phi_{\xi} (u) )  u^{r-1} du \right] = \e^{(r)} \left[ \int_{1}^{m_0} Q_k ( \phi_{\xi} (u) )  u^{r-1} du \right].
\end{equation}
\end{proposition} 
As an application of Theorem \ref{thm harmonic moments Zn} we get a large deviation result.
Indeed, $\e [Z_n^{\lambda}]= \e [e^{\lambda \log Z_n}]$ is the Laplace transform of $\log Z_n$. From Theorem \ref{thm harmonic moments Zn} we obtain
\begin{equation}
\label{eq chi intro}
\lim_{n \to \infty} \frac{1}{n} \log \e_k [Z_n^{\lambda}] = \chi_k (\lambda) = 
\left\{ 
\begin{array}{l l}
\log \gamma_k & \text{if} \ \lambda \leq \lambda_k, \\ 
\Lambda (\lambda) & \text{if} \ \lambda \in [ \lambda_k, 0]. 
\end{array}
\right. 
\end{equation}
Thus using a version of the  G\"artner-Ellis theorem adapted to the study of tail probabilities (see \cite[Lemma 3.1]{liu}), we get the following lower large deviation result for the BPRE $(Z_n)$. Recall that $\Lambda ( \lambda )= \log \e e^{\lambda X}$ is the log-Laplace of $X= \log m_0$, and $\Lambda^* ( \cdot )$ is the Fenchel-Legendre transform of $\Lambda (\cdot)$ defined in \eqref{Lambda*}.
\begin{theorem}
\label{thm PGD}
Let $k \geq 1$ and $r_k$ be the solution of the equation \eqref{equation rj}. Assume that $\e m_0^{r_k+ \varepsilon} < \infty$ for some $\varepsilon>0$. Then,  for any $\theta \in (0, \e[X])$, we have
\begin{equation}
\label{eq PGD GLM}
 \lim_{n \to \infty} -\frac{1}{n} \log \p_k \left(  Z_n \leq e^{\theta n }\right)  = \chi_k^* (\theta) \in (0, \infty),
\end{equation}
where
\begin{eqnarray}
\label{eq chi thm}
\chi_k^* (\theta)  
&=& \sup_{\lambda \leq 0} \left\{ \lambda \theta - \chi_k (\lambda) \right\} 
=
\left\{ 
\begin{array}{cll}
-r_k \theta - \log \gamma_k & \text{if} & 0 < \theta < \theta_k, \\ 
\Lambda^* (\theta) &  \text{if} & \theta_k \leq \theta < \e [X],
\end{array}
\right.
\end{eqnarray}
with
\begin{equation}
\label{theta_k}
\theta_k = \Lambda ' (-r_k).
\end{equation}
\end{theorem}
The value $\chi_k^* (\theta)$ can be interpreted  geometrically as the maximum distance between the graphs of the linear function $l_{\theta} : \lambda \mapsto \theta \lambda$ with slope $\theta$ and the function $\chi_k: \lambda \mapsto \chi_k (\lambda)$ defined in \eqref{eq chi intro}. Taking into account the fact that $\chi (\lambda) = \Lambda (\lambda) $ for $\lambda \in [-r_k, 0]$ and $\chi (\lambda) = \log \gamma_k$ for $\lambda \leq -r_k$, we can easily describe the phase transitions of $\chi^* (\theta)$ depending on the value of the slope $\theta$ of the function $l_\theta$:
\begin{enumerate}
\item in the case when $\theta \in (\theta_k, \e [X])$,  the maximum $\sup_{\lambda \leq 0} \{ l_{\theta} (\lambda) - \chi (\lambda) \}$ is attained for $\lambda \in (-r_k, 0)$ such that $\chi'( \lambda_\theta ) = \Lambda'(\lambda_\theta) = \theta$, whose value is $\chi^*(\theta)= \Lambda^* (\theta)$ (see Fig. 1);
\item the case when $\theta =\theta_k$ is the critical slope for which  the equation $\Lambda'(\lambda) = \theta_k$ has a solution given by $\lambda= -r_k$ (see Fig. 2);
\item in the case when $\theta \in (0, \theta_k)$, the maximum $\sup_{\lambda \leq 0}\{ l_{\theta} (\lambda) - \Lambda (\lambda) \}$ is attained for $\lambda = -r_k$, and then $\chi^*(\theta)= -r_k \theta - \log \gamma_k$ becomes linear  in $\theta$ (see Fig. 3).
\end{enumerate}

\begin{center}
Fig. 1,2,3: Geometrical interpretation of $\chi^* (\theta)$
\end{center}
\vspace{0.5cm}
\begin{minipage}[c]{8cm}
\begin{center}
\begin{tikzpicture}[scale=0.5]
\draw[ ->](-8,0)--(5,0) node[below]{$\lambda$};
\draw[->](0,-4.5)--(0,4.5) ;
\draw [thick, domain=-8:-5] plot(\x, {-5*ln(2)}); 
\draw [thick, domain=-5:3] plot(\x, {-5*ln(1-\x/5)}); 
\draw [ very thick, dashed, domain=-8:3] plot(\x, {-5*ln(1-\x/5)}); 

\draw [domain=-6:4] plot(\x, {0.8*(\x)}); 
\draw (3,2) node[right]{$l_{\theta} (\lambda)$}; 
\draw [very thick] (-2.5, -0.8) node[left]{ \textbf{$\chi^*(\theta)$}};
\draw [->] (-2.5, -0.8)--(-1.2,-0.8);

\foreach \x in {-1.2} \draw[ thick] (\x, 2pt)--(\x,-2pt);
\draw (-1.2,0) node[above]{$\lambda_{\theta}$};
\foreach \x in {-5} \draw[ thick] (\x, 2pt)--(\x,-2pt);
\draw (-5,0) node[above]{$-r_k$};
\draw (0,-3.5) node{--}; 
\draw (0,-3.5) node[right]{$\log \gamma_k$}; 
\draw  (-7,3) node{$\Lambda (\lambda)$};
\draw [ dashed, thick] (-6,3)--(-5,3) ;
\draw  (-7,4) node{$\chi (\lambda)$};
\draw [ thick] (-6,4)--(-5,4) ;
\draw  (-1,-6) node{Fig. 1:  $\theta \in ( \theta_k , \e[X])$};
\end{tikzpicture}
\end{center}
\end{minipage}
\begin{minipage}[c]{7cm}
\begin{flushright}
\begin{tikzpicture}[scale=0.5]
\draw[ ->](-8,0)--(5,0) node[below]{$\lambda$};
\draw[->](0,-4.5)--(0,4.5) ;
\draw [thick, domain=-8:-5] plot(\x, {-5*ln(2)}); 
\draw [thick, domain=-5:3] plot(\x, {-5*ln(1-\x/5)}); 
\draw [ very thick, dashed, domain=-8:3] plot(\x, {-5*ln(1-\x/5)}); 

\draw [ domain=-8:5] plot(\x, {1/2*(\x)}); 
\draw (3,1.1) node[right]{$l_{\theta_k}(\lambda)$};
\draw [very thick, domain=0:1] plot(-5, {-5*ln(2)-(1/2*(5)-5*ln(2))*\x});
\draw [ very thick] (-2.5, -2.7) node[right]{\textbf{$\chi^*(\theta_k)$}};
\draw [->] (-2.5, -2.7)--(-4.9, -2.7) ;

\foreach \x in {-5} \draw[ thick] (\x, 2pt)--(\x,-2pt);
\draw (-5,0) node[above]{$-r_k$};
\draw (0,-3.5) node{--}; 
\draw (0,-3.5) node[right]{$\log \gamma_k$}; 
\draw  (-7,3) node{$\Lambda (\lambda)$};
\draw [ dashed, thick] (-6,3)--(-5,3) ;
\draw  (-7,4) node{$\chi (\lambda)$};
\draw [ thick] (-6,4)--(-5,4) ;
\draw  (-1,-6) node{Fig. 2: $\theta = \theta_k$};
\end{tikzpicture}
\end{flushright}
\end{minipage}

\vspace{0.5cm}

\begin{minipage}[c]{8cm}
\begin{center}
\begin{tikzpicture}[scale=0.5]
\draw[ ->](-8,0)--(5,0) node[below]{$\lambda$};
\draw[->](0,-5)--(0,5) ;
\draw [thick, domain=-8:-5] plot(\x, {-5*ln(2)}); 
\draw [thick, domain=-5:3] plot(\x, {-5*ln(1-\x/5)}); 
\draw [ very thick, dashed, domain=-8:3] plot(\x, {-5*ln(1-\x/5)}); 

\draw [domain=-8:5] plot(\x, {1/4*(\x)}); 
\draw (3,1.7) node[right]{$l_{\theta}(\lambda)$};
\draw [ very thick, domain=0:1] plot(-5, {-5*ln (2) - (-5*ln(2)+5/4)*\x}); 
\draw [very thick] (-5, -2.3) node[left]{\textbf{$\chi^*(\theta)$}};

\foreach \x in {-5} \draw[ thick] (\x, 2pt)--(\x,-2pt);
\draw (-5,0) node[above]{$-r_k$};
\draw (0,-3.5) node{--}; 
\draw (0,-3.5) node[right]{$\log \gamma_k$}; 
\draw  (-7,3) node{$\Lambda (\lambda)$};
\draw [ dashed, thick] (-6,3)--(-5,3) ;
\draw  (-7,4) node{$\chi (\lambda)$};
\draw [ thick] (-6,4)--(-5,4) ;
\draw  (-1,-6) node{Fig. 3:  $\theta \in (0, \theta_k)$};
\end{tikzpicture}
\end{center}
\end{minipage}
\vspace{0.5cm}

\begin{remark}
Theorem \ref{thm PGD} corrects and improves the result of  \cite[Theorem 3.1(ii)]{bansaye2013lower}.
Moreover it gives new and alternative expressions of the rate function 
and the critical value. Actually it was proved in \cite[Theorem 3.1(ii)]{bansaye2013lower} that, assuming $\p (Z_1=0)=0$ and  $\e m_0^t< \infty$ for all $t>0$, we have 
\begin{equation}
\label{eq PGD GLM bansaye}
 \lim_{n \to \infty} -\frac{1}{n} \log \p_k \left(  Z_n \leq e^{\theta n }\right)  =I_k (\theta) \in (0, \infty),
\end{equation}
where
\begin{eqnarray}
\label{eq chi thm002}
I_k ( \theta)  
&=&
\left\{ 
\begin{array}{ccl}
\rho_k \left(1 - \frac{\theta}{\theta_k^*} \right) + \frac{\theta}{\theta_k^*}
\Lambda^* (\theta_k^*) & \text{if} & 0 < \theta \leq \theta_k^*, \\ 
\Lambda^* (\theta) &  \text{if} & \theta_k^* \leq \theta < \e [X],
\end{array}
\right.
\end{eqnarray}
with 
\begin{equation}
\label{rho_k bansaye}
\rho_k = \lim_{n \to \infty} -\frac{1}{n} \log \p_k ( Z_n = j )
\end{equation}
and 
$\theta_k^* $ the unique solution on $(0, \e [X])$ of the equation 
\begin{equation}
\label{theta_* bansaye}
\frac{\rho_k- \Lambda^* (\theta_k^*)}{\theta_k^*} = \inf_{0 \leq \theta \leq \e [X]} \frac{\rho_k- \Lambda^* (\theta)}{\theta}.
\end{equation}

It has been stated mistakenly in \cite{bansaye2014small} that $\rho_k = -k \log \gamma$, whereas the correct statement is 
\begin{equation}
\label{eq correction rho_k}
\rho_k = - \log \gamma_k ,
\end{equation}
according to \cite[Theorem 2.3]{glm2016asymptotic} (see Lemma \ref{thm small value probability 2} below).
With this correction, the two critical values $\theta_k$ and $\theta_k^*$ and the two rate functions $I_k$ and $\chi_k^*$ coincide, that is 
\begin{equation}
\label{eq expression alternative fct taux}
\theta_k = \theta_k^* \quad \text{and} \quad \chi_k^* (\theta) = I_k (\theta) \quad \text{for all}\ \ \theta \in (0, \e [X]) .
\end{equation}
Indeed, by the definition of $\theta_k^*$, the derivative of the function $ \theta \mapsto \frac{\rho_k - \Lambda^* (\theta)}{\theta}$ vanishes for $\theta=\theta_k^*$. Therefore,  since $(\Lambda^*)' (\theta) = \lambda_{\theta} $ with $\Lambda'(\lambda_{\theta})=\theta$, we get, for $\theta=\theta_k^*$, 
\begin{equation}
\label{eq theta=theta_k^*}
 \Lambda^* (\theta)=  \lambda_{\theta} \theta + \rho_k.
\end{equation}
Using the identity $\Lambda^*(\theta) = \lambda_{\theta} \theta - \Lambda (\lambda_\theta)$, we obtain
\[ 
\Lambda (\lambda_\theta) = - \rho_k,
\]
which implies that $\lambda_{\theta} = -r_k$ and then $\theta_k^* = \Lambda ' (-r_k) = \theta_k$.

Moreover, coming back to \eqref{eq theta=theta_k^*} and using the identities
$\Lambda (-r_k) = \log \gamma_k = - \rho_k$ and $\theta_k = \Lambda ' (-r_k)$, we get
\[
- r_k \theta_k - \Lambda (-r_k) = - r_k \theta_k + \rho_k = \Lambda_k^* (\theta_k).
\] 
Therefore, for any $\theta \in [0, \theta_k]$,
\begin{eqnarray*}
-r_k \theta - \log \gamma_k  &=& -r_k \theta - \Lambda (-r_k) \\
&=& \frac{\theta}{\theta_k} \left(-r_k \theta_k - \Lambda (-r_k) \right) +\frac{\theta}{\theta_k} \Lambda (-r_k) - \Lambda (-r_k) \\
&=& \frac{\theta}{\theta_k} \Lambda^* ( \theta_k ) - \left( 1 - \frac{\theta}{\theta_k} \right) \log \gamma_k,
\end{eqnarray*}
so that  $I_k (\theta) = \chi_k^* (\theta)$, which ends the proof of \eqref{eq expression alternative fct taux}.
From \eqref{eq expression alternative fct taux} and Theorem \ref{thm harmonic moments Zn}, we see that \eqref{eq PGD GLM bansaye} is valid assuming only $\p (Z_1=0)=0$ and  $\e [ m_0^{r_k+ \varepsilon} ] < \infty$ for some $\varepsilon>0$.
Actually when $\p (Z_1=0)>0$, as shown in \cite[Theorem 3.1 (i)]{bansaye2013lower},   \eqref{eq PGD GLM bansaye} remains valid with $\rho_k = \lim_{n \to \infty} -\frac{1}{n} \log \p_k ( Z_n = j ) = \rho>0$ independent of $k$. 
\end{remark}

Similarly, one can apply Theorem \ref{thm harmonic moments Zn} to get the decay rate for the probability $\p (Z_n \leq k_n)$, where $k_n $ is any sub-exponential sequence in the sense that $k_n \to \infty$ and $ k_n / \exp (\theta n) \to 0$ for every $\theta>0$, as stated in the following corollary.
\begin{corollary}
\label{cor LD ss exp}
Let $k \geq 1$ and assume that $\e [m_0^{r_k+\varepsilon}]<\infty$ for some $\varepsilon>0$. 
Let $k_n>0$ be such that $k_n \to \infty$ and $ k_n / \exp (\theta n) \to 0$ for every $\theta>0$, as $n \to \infty$. Then
\begin{equation}
\label{asympt dev kn}
\lim_{n\to\infty}\frac{1}{n} \log \mathbb{P}_k \left( Z_n \leq k_n \right) = \log \gamma_k.
\end{equation}
\end{corollary}
It was stated mistakenly in \cite[Theorem 3.1(ii)]{bansaye2013lower} that $
\lim_{n\to\infty} \frac{1}{n} \log \mathbb{P}_k \left( Z_n \leq k_n \right) =  k \log \gamma $. To show \eqref{asympt dev kn}, it suffices to note that by  Markov's inequality and Theorem \ref{thm harmonic moments Zn}, we have, for any $r > r_k$, 
$$ \gamma_k^n = \p_k (Z_n = k ) \leq \mathbb{P}_k ( Z_n \leq k_n ) \leq \e [Z_n^{-r}] k_n^{r} \leq \min \{ \gamma_k^n k_n^{r}, \gamma_k^n k_n^{r_k} n \}. $$ 
 
The above argument leads to a precise large deviation bound as stated below.
\begin{corollary}
\label{cor LD borne markov}
Let $k \geq 1$ and assume that $\e [m_0^{r_k+\varepsilon}]<\infty$ for some $\varepsilon>0$. Then
\begin{equation}
\label{eq sharp lower large deviation bound}
\p_k \left( Z_n \leq e^{\theta n } \right) \leq \inf_{r >0} \frac{\e_k \left[ Z_n^{-r} \right]}{e^{\theta r n}} = \left\{ \begin{array}{l c l} 
\displaystyle{ e^{ - n (  -\theta r_k - \Lambda (-r_k)) }}  &  \text{if} & 0 < \theta \leq \theta_k, \\
& \\
\displaystyle{n e^{ - n (  - \theta_k r_k - \Lambda (-r_k))} } &  \text{if} &  \theta = \theta_k, \\
& \\
\displaystyle{ e^{ - n  \Lambda^* (\theta) } } &  \text{if}& \theta_k \leq \theta < \e [X].
\end{array}
\right.
\end{equation}
\end{corollary}
The question of the exact decay rate of $\p (Z_n \leq e^{\theta n } )$ will be treated in a forthcoming paper.

As an example, let us consider the case where the reproduction law has a fractional linear generating function, that is when
\begin{equation}
\label{loi linéaire fractionnaire}
p_0 (\xi_0) = a_0 \quad \text{and}  \quad p_k (\xi_0) = \frac{(1-a_0)(1-b_0)}{b_0} b_0^k \quad \text{for all}\ k \geq 1,
\end{equation}
for which the generating function is
\begin{equation*}
\label{eq f cas lineraire fractionnaire}
f_0(t) = a_0 + \frac{(1-a_0)(1-b_0)t}{1-b_0t},
\end{equation*}
where $a_0\in [0,1) and $ $b_0 \in  (0,1)$ are random variables depending on the environment $\xi_0$. 
This case has been examinated by several authors (see e.g. \cite{kozlov2006large, 	Nakashima2013lower}).
In the case where $a_0=0$ (non-extinction), the BPRE is said to be geometric; in this case $X= \log m_0 = - \log (1-b_0)$, $\log \gamma_k = \log \e [e^{-k X}]= \Lambda (-k)$ and $r_k = k$. Therefore, we obtain the following explicit version of Theorem \ref{thm harmonic moments Zn}:
\begin{corollary}
\label{cor LDP Geom}
Let $Z_n$ be a geometric BPRE. Assume that there exists $\varepsilon>0$ such that $\e [e^{(k+\varepsilon)X}] < \infty$. Then \eqref{eq PGD GLM} holds with
\begin{eqnarray}
\label{eq chi_k geometrique}
\chi^*_k (\theta)  
&=&
\left\{ 
\begin{array}{ccl}
-k \theta - \log \e [e^{-k X}] & \text{if} & 0 < \theta \leq \theta_k, \\ 
\Lambda^* (\theta) &  \text{if} &  \theta_k \leq \theta < \e [X],
\end{array}
\right.
\end{eqnarray}
where
\begin{equation}
\label{eq theta_k geometrique}
\theta_k = \e [ X e^{-k X}] / \e [e^{-k X}].
\end{equation}
\end{corollary}
Corollary \ref{cor LDP Geom} recovers and completes the large deviation result in \cite[Corollary 3.3]{bansaye2013lower} for a fractional linear BPRE with $a_0 >0$, where \eqref{eq PGD GLM} was obtained  with
\begin{eqnarray}
\label{eq chi_k geometrique}
I (\theta)  
&=&
\left\{ 
\begin{array}{ccl}
- \theta - \log \e [e^{- X}] & \text{if} & 0 < \theta \leq \theta^*, \\ 
\Lambda^* (\theta) &  \text{if} &  \theta^* \leq \theta < \e [X],
\end{array}
\right.
\end{eqnarray}
and
\begin{equation}
\label{eq theta_k geometrique}
\theta^*= \e [ X e^{- X}] / \e [e^{-X}].
\end{equation}
In fact the result in \cite[Corollary 3.3]{bansaye2013lower} was stated without the hypothesis $a_0>0$, but 
when $a_0=0$, this result is valid only for $k=1$,  as shown by Corollary \ref{cor LDP Geom} (for $k \geq2$, the factor $k$ is missing in \cite[Corollary 3.3]{bansaye2013lower}).

As another consequence of Theorem \ref{thm harmonic moments Zn}, we improve an earlier result about the rate of  convergence in the central limit theorem for $W-W_n$.
\begin{theorem}
\label{thm TCL W-Wn}
Assume $ \text{essinf}\ \frac{m_0 (2)}{m_0^2} >1$ and $\mathbb{E} Z_1^{2+ \varepsilon} < \infty$ for some $\varepsilon \in (0, 1]$. Then there exists a constant $C>0$ such that, for all $k \geq 1$,
\begin{equation}
\label{convergence W-Wn}
\sup_{x \in \mathbb{R}} \left| \mathbb{P}_k \left( \frac{\Pi_n ( W- W_n)}{\sqrt{Z_n} \delta_{\infty} (T^n \xi )} \leq x \right) - \Phi (x) \right| \leq C A_{k,n} ( \varepsilon/2),
\end{equation}
where $\Phi (x)= \frac{1}{\sqrt{2 \pi}}\int_{- \infty}^{x} e^{-t^2/2} dt$ is the standard normal distribution function.
\end{theorem}
Theorem \ref{thm TCL W-Wn} improves the exponential rate of convergence in \cite[Theorem 1.7]{liu}. In our approach 
the assumption $ \text{essinf}\ \frac{m_0 (2)}{m_0^2} >1$ is required to ensure that the quenched variance $\delta_{\infty} (\xi)$ of the variable $W$ is a.s. separated from $0$. This hypothesis does not seem natural and should be relaxed. One should be able to find a suitable hypothesis to ensure the existence of harmonic moments for the random variable $\delta_{\infty} (\xi)$, which would be enough for our objective.

As another consequence of Theorem \ref{thm harmonic moments Zn}, we give some  large deviation results on the ratio 
\begin{equation}
\label{eq Rn}
R_n = \frac{Z_{n+1}}{Z_n}
\end{equation} 
toward the conditional mean $m_n = \sum_{k=1}^{\infty} k p_k (\xi_n)$. 
Let 
\begin{equation}
M_{n,j} = j^{-1} \sum_{i=1}^{j} N_{n,i} 
\label{Mnj-001}
\end{equation}
be the empirical mean of $m_n$ of size $j$ under the environment $\xi$, where the r.v.'s $N_{n,i} \ (i=1, \ldots , j)$ are i.i.d.\ with generating function $f_n$.
\begin{theorem}
\label{thm LDR Rn 1}
Let $k \geq 1$. If, for some set $D \subset \mathbb{R}$, there exist some constants $C_1>0$ and $r>0$ such that, for all $j \geq 1$,
\begin{equation}
\label{majoration LGN MA}
\mathbb{P} \left( M_{0, j} -m_0 \in D \right) \leq \frac{C_1}{j^r},
\end{equation}
then there exists a constant $B_1 \in (0, \infty)$ such that for all $n \geq 1$,
\begin{equation}
\label{majoration limsum Rn}
\mathbb{P}_k \left( R_n - m_n \in D \right) \leq B_1 A_{k,n} (r),
\end{equation}
where $A_{k,n} (r)$ is defined in \eqref{An}.
Similarly, if  there exist some constants $C_2>0$ and $r>0$ such that, for all $j \geq 1$,
\begin{equation}
\label{minoration LGN MA}
\mathbb{P} \left(  M_{0, j} - m_0 \in D \right) \geq \frac{C_2}{j^r},
\end{equation}
then there exists a constant $B_2 \in (0, \infty)$ such that for all $n \geq 1$,
\begin{equation}
\label{majoration liminf Rn}
 \mathbb{P}_k \left( R_n - m_n \in D \right) \geq B_2 A_{k,n} (r).
\end{equation}
\end{theorem}
This  result shows that there exist some phase transitions in the rate of convergence depending on whether the value of $r$ is less than, equal or greater than the constant $\gamma_k$. 
The next result gives a bound of the large deviation probability of $R_n-m_n$ under a simple moment condition on $Z_1$.
\begin{theorem}
\label{thm LDR Rn 2}
Let $k \geq 1$. Assume that there exists $p>1$ such that $ \e  |Z_1-m_0|^p < \infty$.
Then, there exists a constant $C_p>0$ such that, for any $a>0$,
\begin{equation}
\p_k (|R_n - m_n |>a) \leq \left\{ \begin{array}{l l l }
C_p a^{-p} A_{k,n} (p-1) & \text{if}&  p \in (1, 2], \\
C_p a^{-p} A_{k,n} (p/2) & \text{if}&  p \in (2, \infty) .
\end{array}
\right.
\label{LDR001}
\end{equation}
\end{theorem}

\section{Proof of main theorems} 
\label{sec preuve main thm}
In this section we will prove the main results of this paper, Theorems \ref{thm harmonic moments Zn} and \ref{thm PGD}, and the associated result, Proposition \ref{prop relation G Q psi BPRE}.
In Section \ref{sec aux results} we present some auxiliary results concerning the critical value for the existence of the harmonic moments of the r.v. $W$ and the asymptotic behavior of the asymptotic distriution $\p_k (Z_n = j)$ as $n \to \infty$, with $ j \geq k \geq 1$.
In Sections \ref{subsec preuve main thm} and \ref{subsec PGD} we prove respectively Theorems \ref{thm harmonic moments Zn} and \ref{thm PGD}.
The proof of Proposition \ref{prop relation G Q psi BPRE} is given in Section \ref{subsec preuve prop GW} for a Galton-Watson process and in Section \ref{subsec preuve prop} for a general BPRE. 

\subsection{Auxiliary results}
\label{sec aux results}
We recall some results to be used in the proofs. The first one concerns the critical value for the existence of harmonic moments of the r.v. $W$.
\begin{lemma}[\cite{glm2016asymptotic}, Theorem 2.1]
\label{thm harmonic moments W}
Assume 
that there exists a constant $p>0$ such that $\mathbb{E} \left[ m_0^{p} \right] < \infty $. Then for any $a\in (0,p)$,
\begin{equation*}
\mathbb{E}_k W^{-a} < \infty \quad \text{if and only if} \quad \mathbb{E} \left[ p_1^k (\xi_0) m_0^a \right] <1.
\end{equation*}
\end{lemma}
The second result is about the asymptotic equivalent of the probability $\p_k (Z_n = j)$ as $n \to \infty$, for any $ j \geq k \geq 1$.
\begin{lemma}[\cite{glm2016asymptotic}, Theorem 2.3]
Assume that $\p(Z_1=1)>0$. For any $k\geq 1$  the following assertions hold.
\begin{enumerate}[ref=\arabic*, leftmargin=*, label=\arabic*.]
\item[a)] \label{thm small value probability 2}
 For any accessible state $ j \geq k$ in the sense that $\mathbb{P}_k (Z_l=j)>0$ for some $l \geq 0$, we have  
\begin{equation}
\label{small value asymptotic 2}
\mathbb{P}_k \left( Z_n = j \right) \underset{n \to \infty}{\sim} \gamma_k^n q_{k,j},
\end{equation}
where $q_{k,k}= 1$ and, for  $j>k$,  $q_{k,j} \in (0, + \infty )$ is the solution of the recurrence relation
\begin{equation}
\label{relation rec qkj}
\gamma_k q_{k,j} = \sum_{i=k}^j p(i, j) q_{k, i},
\end{equation}
with $q_{k,i}=0$ for any non-accessible state $i$ (i.e.\  $\mathbb{P}_k (Z_l=i)=0$ for all $l \geq 0$).
\item[b)] Assume 
that  there exists $\varepsilon>0$ such that $\mathbb{E} [ m_0^{r_k+ \varepsilon} ] < \infty $. Then, for any $r>r_k$, we have 
\begin{equation}
\label{eq decroissance qkj}
\sum_{j=k}^{\infty} j^{-r} q_{k,j} < \infty. 
\end{equation}
In particular, the radius of convergence of the power series
\begin{equation}
Q_k (t) = \sum_{j=k}^{+ \infty} q_{k,j} t^j
\end{equation}
is equal to 1. 
\item[c)] For all $t \in [0, 1)$ and $k \geq 1$, we have,
\begin{equation}
\label{cv Qnk ->Qk}
 \frac{G_{k, n} (t)}{\gamma^n_k} \uparrow Q_k(t) \ \ \text{as} \ \ n \to \infty,
\end{equation}
where $G_{k,n}$ is the probability generating function of $Z_n$ when $Z_0=k$, defined in \eqref{Gkn}.
\item[d)] $Q_k (t)$ is the unique power series which verifies the functional equation
\begin{equation}
\label{relation Q_k}
\gamma_k Q_k (t) = \mathbb{E} \left[ Q_k ( f_0 (t) ) \right], \ \ t\in [0,1),
\end{equation}
with the condition $Q_k^{(k)} (0) = 1.$
\end{enumerate}
\end{lemma}

\subsection{Proof of Theorem \ref{thm harmonic moments Zn}}
\label{subsec preuve main thm}
In this section we give a proof of the convergence of the 
normalized harmonic moments $\mathbb{E}_k[Z_n^{-r}] / A_{k,n}(r) $ as $n \to \infty,$ where $A_{k,n}(r)$ is defined in \eqref{An}.

a) We first consider the case when $r<r_k$  (which corresponds to the case $\gamma_k < c_r$).
By the change of measure \eqref{changement de mesure}, we obtain
\begin{equation}
\label{changement mesure EZn = EWn cr}
\e_{k} \left[ Z_n^{-r} \right] = \mathbb{E}_k^{(r)} [ W_n^{-r} ] c_r^n,
\end{equation}
with $c_r = \mathbb{E} m_0^{-r}$.
From \eqref{cv L1 Pr W} and \cite[Lemma 2.1]{liu}, it follows that the sequence $(\e_k^{(r)} [ W_n^{-r} ] )$ is increasing and
\begin{equation}
\label{EWn croissant 1}
\lim_{n \to \infty} \mathbb{E}_k^{(r)} [ W_n^{-r} ] = \sup_{n \in \mathbb{N}} \mathbb{E}_k^{(r)} [ W_n^{-r} ] = \mathbb{E}_k^{(r)} [ W^{-r} ].
\end{equation}
Moreover, for any $r<r_k$, we have $\gamma_k < c_r$, which implies that $\mathbb{E}^{(r)} [ p_1^{k} (\xi_0) m_0^{r} ] = \gamma_{k}/ \e m_0^{-r}<1$.
So, by Lemma \ref{thm harmonic moments W}, we get, for any $r<r_k$,
\begin{equation}
\label{W-r existe 1}
\mathbb{E}_k^{(r)} [ W^{-r} ]  < \infty .
\end{equation}
Therefore, coming back to \eqref{changement mesure EZn = EWn cr} and using \eqref{EWn croissant 1} and \eqref{W-r existe 1}, we obtain
\begin{equation}
\label{eq C(k,r) r<r_k}
\mathbb{E}_k \left[ Z_{n}^{-r} \right] c_r^{-n} \underset{ n \to \infty}{\uparrow} \mathbb{E}_k^{(r)} \left[ W^{-r} \right] \in (0, \infty).
\end{equation}
To give an integral expression of the limit constant $C(k,r)$, we shall use the following expression for the inverse of a positive random variable $X^r$: for any $r>0$, we have
\begin{equation}
\label{eq expression X^-r}
\frac{1}{X^r} = \frac{1}{\Gamma (r)} \int_0^{+ \infty}  e^{-u X}  u^{r-1} du.
\end{equation}
Then, from \eqref{eq C(k,r) r<r_k}, \eqref{eq expression X^-r} and Fubini's theorem,  we get
\begin{equation}
\label{eq C(k,r) r<r_k 2}
\mathbb{E}_k \left[ Z_{n}^{-r} \right] c_r^{-n} \underset{ n \to \infty}{\uparrow} \frac{1}{\Gamma(r)} \int_0^{\infty} \phi_k^{(r)} (u) u^{r-1} du,
\end{equation}
which proves \eqref{equivalent Z_n} for $r<r_k$.

b) Next we consider the case when $r>r_k$ (which corresponds to the case $ \gamma_k > c_r$).
Using parts a) and b) of Lemma \ref{thm small value probability 2} and the monotone convergence theorem, it follows that
\begin{eqnarray}
\label{eq C(k,r) r>r_k}
\lim_{n \to \infty} \uparrow \frac{\e_k [Z_n^{-r}]}{\gamma_k^n} &=& \lim_{n \to \infty} \uparrow \sum_{j=k+1}^{\infty} k^{-r} \frac{\p (Z_n =k)}{\gamma_k^n} \notag \\
&=&  \sum_{j=k+1}^{\infty} k^{-r} q_{k,j} < \infty.
\end{eqnarray}
Moreover, using \eqref{eq expression X^-r} together with Fubini's theorem and the change of variable $ju=t$, we obtain
\begin{eqnarray*}
\frac{1}{\Gamma (r)} \int_0^1  Q_k (e^{-u} ) u^{r-1} du &=& \sum_{j=k}^{\infty} q_{k,j} j^{-r} \frac{1}{\Gamma (r)} \int_0^{\infty} e^{-t} t^{r-1} dt = \sum_{j=k}^{\infty} q_{k,j} j^{-r}.
\end{eqnarray*}
Therefore, coming back to \eqref{eq C(k,r) r>r_k}, we get
\begin{equation}
\label{eq C(k,r) r>r_k 2}
\lim_{n \to \infty} \uparrow \frac{\e_k [Z_n^{-r}]}{\gamma_k^n} = \frac{1}{\Gamma (r)} \int_0^1  Q_k (e^{-u} ) u^{r-1} du,
\end{equation}
which proves \eqref{equivalent Z_n} for $r>r_k$.

c) Now we consider the case when $r=r_k$ (which corresponds to $r= \gamma_k$). For $n \geq 1$ and $m \geq 0$, we have the following well-known branching property for $Z_n$:
\begin{equation}
\label{decomposition Zn1}
Z_{n+m} = \sum_{i=1}^{Z_m} Z_{n,i}^{(m)},
\end{equation}
where the r.v.'s $Z_{n,i}^{(m)}$ $( i \geq 1) $ are independent of $Z_m$ under $\mathbb{P}_{\xi}$ and $\p$.
Moreover, under $\p_\xi$, for each $n \geq 0$, the r.v.'s $Z_{n,i}^{(m)}$ $( i \geq 1) $ are i.i.d.\ with the same conditional probability law  $ \mathbb{P}_{\xi} \left( Z_{n,i}^{(m)} \in \cdot \right)= \mathbb{P}_{T^m \xi} \left( Z_n \in \cdot \right) $, where  $T^m$ is the shift operator defined by $T^m (\xi_0, \xi_1 , \ldots ) = (\xi_m, \xi_{m+1} , \ldots )$. 
Intuitively, relation \eqref{decomposition Zn1} shows that, conditionally on $Z_m=i$, the annealed law of the process $Z_{n+m}$ 
is the same as that of a new process $Z_n $ starting with $i$ individuals.
Using  \eqref{decomposition Zn1} with $m=1$, 
we obtain
\begin{eqnarray}
\label{eq induction 1}
\mathbb{E}_k \left[ Z_{n+1}^{-r} \right] 
&=& \mathbb{E}_k \left[ Z_{n}^{-r}\right] \p_k ( Z_1 = k) + \sum_{i=k+1}^{\infty} \mathbb{E}_i \left[ Z_n^{-r}  \right] \p_k ( Z_1 =i) .
\end{eqnarray}
From \eqref{changement mesure EZn = EWn cr}, we have
$
\e_{i} \left[ Z_n^{-r} \right] =\mathbb{E}_i^{(r)} [ W_n^{-r} ] c_r^n .
$
Substituting this into \eqref{eq induction 1} and setting 
\[
b_{n} = \sum_{i=k+1}^{\infty} \mathbb{E}_i^{(r)} [ W_n^{-r} ] \p_k ( Z_1 =i) ,
\] we get
\begin{equation}
\label{eq iteration Zn}
\e_k [ Z_{n+1}^{-r}] = \e_k [ Z_n^{-r}] \gamma_k + b_n c^n_r,
\end{equation}
with $\gamma_k = \p_k (Z_1=k)$. Iterating \eqref{eq iteration Zn} leads to
\begin{equation}
\label{induction exacte EZn-r}
\mathbb{E}_{k} \left[ Z_{n+1}^{-r} \right] = \gamma_{k}^{n+1}k^{-r}  +  \sum_{j=0}^{n} \gamma_{k}^{n-j} b_j c_r^j.
\end{equation}
Using the fact that $r=r_k$ (which corresponds to $\gamma_k=c_{r}$) and dividing \eqref{induction exacte EZn-r} by $\gamma^{n+1}n$, we get
\begin{eqnarray}
\label{eq cas r=rk}
\frac{\mathbb{E}_k \left[ Z_{n+1}^{-r} \right]}{n \gamma^{n+1}} = \frac{k^{-r}}{n} + \frac{\gamma_k^{-1}}{n} \sum_{j=0}^{n} b_j.
\end{eqnarray}
To prove the convergence in \eqref{eq cas r=rk} we need to show that $\lim_{n \to \infty} b_n < \infty.$ By \eqref{EWn croissant 1} and the monotone convergence theorem, we have
\begin{equation}
\label{bn}
b:= \lim_{n \to \infty} \uparrow b_n = \sum_{i=k+1}^{\infty} \mathbb{E}_i^{(r)} [ W^{-r} ] \p_k ( Z_1 =i).
\end{equation}
Now we show that $b < \infty$. Using \eqref{decomposition Zn1} for $m=0$ and $Z_0=i$, with $i >k$, and the fact that $Z_{n,j} >0$ for all $1 \leq j \leq i$, we obtain
\begin{equation} 
\label{eq Ei leq Ek}
\e_{i}[Z_n^{-r}]= \e \left[ \left( Z_{n,1} + \ldots + Z_{n,i} \right)^{-r} \right] \leq \e \left[ \left( Z_{n,1} + \ldots + Z_{n,k} \right)^{-r} \right] = \e_{k}[Z_n^{-r}].
\end{equation}
By \eqref{eq Ei leq Ek} and the change of measure \eqref{changement mesure EZn = EWn cr}, we get, for any $i \geq k+1$,
\[ 
\mathbb{E}_i^{(r)} [W_n^{-r}] \leq \mathbb{E}_{k+1}^{(r)} [W_n^{-r}].
\]
Then, as in \eqref{EWn croissant 1}, letting $n \to \infty$ leads to
\begin{equation}
\label{ Ei W leq Ek+1 W}
\mathbb{E}_i^{(r)} [W^{-r}] \leq \mathbb{E}_{k+1}^{(r)}[W^{-r}].
\end{equation}
Now we shall prove that 
\begin{equation}
\label{W-r existe 2}
\mathbb{E}_{k+1}^{(r)} [ W^{-r} ] < \infty.
\end{equation} 
For this it is enough to verify the condition of Lemma \ref{thm harmonic moments W} under the measure $\e_k^{(r)}$ defined by \eqref{changement de mesure}:
indeed, for $r=r_k$ we have $\gamma_k= c_r$, which implies that
\[
\mathbb{E}^{(r)} [ p_1^{k+1} (\xi_0) m_0^{r} ] = \frac{\gamma_{k+1}}{\e m_0^{-r}} = \frac{\gamma_{k+1}}{\gamma_k}<1.
\]
This proves \eqref{W-r existe 2}. Using \eqref{ Ei W leq Ek+1 W} and \eqref{W-r existe 2}, we obtain
\[
b = \sum_{i=k+1}^{\infty} \mathbb{E}_i^{(r)} [ W^{-r} ] \p_k ( Z_1 =i) \leq \mathbb{E}_{k+1}^{(r)} [ W^{-r} ] \sum_{j=k+1}^{\infty} \p_k (Z_1=j)  < \infty.
\]
Therefore, coming back to \eqref{eq cas r=rk}, using \eqref{bn} and Cesaro's lemma, we get
\begin{equation}
\label{eq C(k,r) r=r_k}
\lim_{n \to \infty} \frac{\mathbb{E}_k \left[ Z_{n+1}^{-r} \right]}{n \gamma_k^{n+1}} = \frac{1}{\gamma_k}\sum_{i=k+1}^{\infty} \mathbb{E}_{i}^{(r)}[ W^{-r} ] \p_k ( Z_1 =i) < \infty,
\end{equation}
which proves \eqref{equivalent Z_n} for $r=r_k$,  with
\[
C(k,r) =\frac{1}{\gamma_k}\sum_{i=k+1}^{\infty} \mathbb{E}^{(r)}_{i} [ W^{-r} ] \p_k ( Z_1 =i).
\]
We now show an integral expression of the constant $C(k,r)$.
Recall that $W$ admits the following decomposition
\begin{equation}
\label{eq decomposition W Z0}
W= \sum_{j=1}^{Z_0} W(j),
\end{equation}
where the r.v.'s $W(j)$   $(j = 1, 2, \dots)$  are independent of $Z_0$ and $m_0$ under $\p_{\xi}$ and $\p$.
Moreover, conditionally on the environement $\xi$,
the r.v.'s $W(j)$   $(j = 1, 2, \dots)$ are i.i.d. with common law  $ \p_{\xi} ( W(j) \in \cdot ) = \p_{\xi} ( W \in \cdot )$. 
With these considerations, it can be easily seen that 
\begin{equation}
\label{eq relation Psik psi}
\phi_i^{(r)} (t) =\e_i^{(r)}[\phi_{\xi} (u)] =
\e^{(r)} [\phi_{\xi} (u)^{i}].
\end{equation}
Therefore, using \eqref{eq expression X^-r} with $r=r_k$, together with \eqref{eq relation Psik psi} and  Fubini's theorem, we obtain
\begin{eqnarray}
\label{expression integrale C(k, rk)}
\frac{1}{\gamma_k}\sum_{i=k+1}^{\infty} \mathbb{E}^{(r)}_{i} [ W^{-r} ] \p_k ( Z_1 =i) &=&  \frac{1}{\gamma_k \Gamma (r)} \sum_{i=k+1}^{\infty} \mathbb{E}^{(r)}_{i} \int_{0}^{\infty} e^{-u W} u^{r-1} du \ \p_k ( Z_1 =i) \notag\\
&=&  \frac{1}{\gamma_k \Gamma (r)} \sum_{i=k+1}^{\infty} \mathbb{E}^{(r)} \int_{0}^{\infty} \phi_{\xi}^i (u) u^{r-1} du \ \p_k ( Z_1 =i) \notag\\
&=&  \frac{1}{\gamma_k\Gamma (r)}  \mathbb{E}^{(r)} \int_{0}^{\infty} \sum_{i=k+1}^{\infty} \phi_{\xi}^i (u) \ \p_k ( Z_1 =i) u^{r-1} du \notag \\
&=&  \frac{1}{\gamma_k \Gamma (r)}  \mathbb{E}^{(r)} \left[ \int_{0}^{\infty} \bar{G}_{k,1} ( \phi_{\xi} (u) )  u^{r-1} du \right],
\end{eqnarray}
where $\bar{G}_{k,1} (u) = {G}_{k,1} (u) - \gamma_k u^k = \sum_{j=k+1}^{\infty } u^j \p (Z_1=j) $.
Therefore, using \eqref{eq C(k,r) r=r_k} and \eqref{expression integrale C(k, rk)}, we get
\[
\lim_{n \to \infty} \frac{\mathbb{E}_k \left[ Z_{n+1}^{-r} \right]}{n \gamma_k^{n+1}} = \frac{1}{\gamma_k \Gamma (r)}  \mathbb{E}^{(r)} \left[ \int_{0}^{\infty} \bar{G}_{k,1} ( \phi_{\xi} (u) )  u^{r-1} du \right],
\]
which ends the proof of Theorem \ref{thm harmonic moments Zn}.

\subsection{Proof of Theorem \ref{thm PGD}}
\label{subsec PGD}

In this section we prove Theorem \ref{thm PGD}. For convenience, let $\lambda_k=-r_k$. From Theorem \ref{thm harmonic moments Zn}, for any $k \geq 1$, we have 
\begin{equation}
\label{eq appl Garnter Ellis}
\lim_{n \to \infty} \frac{1}{n} \log \e_k [Z_n^{\lambda}] = \chi_k (\lambda) = 
\left\{ 
\begin{array}{l l}
\log \gamma_k & \text{if} \ \lambda \leq \lambda_k, \\ 
\Lambda (\lambda) & \text{if} \ \lambda \in [ \lambda_k, 0]. 
\end{array}
\right. 
\end{equation}
Thus using a version of the  G\"artner-Ellis theorem adapted to the study of tail probabilities (see \cite[Lemma 3.1]{liu}) and the fact that $\chi_k (\lambda) = \log \gamma_k$ for all $\lambda \leq \lambda_k$, we obtain, for all $\theta \in (0, \e [X])$,
\begin{equation}
\lim_{n \to \infty} - \frac{1}{n} \p_k ( Z_n \leq e^{\theta n} ) = \chi^* (\theta),
\end{equation}
with
\begin{equation}
\label{Gamma * 000}
\chi_k^*(\theta) = \sup_{\lambda \leq 0} \left\{ \lambda \theta - \chi_k (\lambda) \right\}
= \max \left\{ \lambda_k \theta - \Lambda (\lambda_k) , \sup_{ \lambda_k \leq \lambda \leq 0} \left\{ \lambda \theta - \Lambda (\lambda) \right\} \right\}.
\end{equation}
It is well-known (see e.g. \cite[Lemma 2.2.5]{dembo1998large}) that the function
\[
\Lambda^* (\theta) = \sup_{\lambda \leq 0} \left\{ \lambda \theta - \Lambda (\lambda) \right\} =  \left\{ \lambda_{\theta} \theta - \Lambda (\lambda_{\theta} ) \right\}, \quad \text{with} \quad \Lambda'(\lambda_{\theta}) = \theta,
\]
is non-increasing for $\theta \in (0, \e [X])$.
Therefore, letting
\begin{equation}
\theta_k=\Lambda'(\lambda_k),
\end{equation}
it follows that:
\begin{enumerate}
\item for any $\theta \in (0, \theta_k]$, 
\[
\lambda_k \theta - \Lambda (\lambda_k) \geq \lambda_k \theta_k - \Lambda (\lambda_k) = \Lambda^* (\theta_k) = \sup_{ \lambda_k \leq \lambda \leq 0} \left\{ \lambda \theta - \Lambda (\lambda) \right\} ;
\]
\item for any $\theta \in [\theta_k, \mu)$,
\[
\Lambda^* (\theta) = \sup_{ \lambda_k \leq \lambda \leq 0} \left\{ \lambda \theta - \Lambda (\lambda) \right\} \geq \Lambda^* (\theta_k) = \lambda_k \theta_k - \Lambda (\lambda_k) \geq \lambda_k \theta - \Lambda (\lambda_k) .
\]
\end{enumerate}
With these considerations, we get from \eqref{Gamma * 000} that
\begin{equation}
\label{eq appl Garnter Ellis}
\chi_k^*(\theta) = 
\left\{ 
\begin{array}{l l}
\lambda_k \theta - \Lambda (\lambda_k) & \text{if} \ \theta \in (0, \theta_k], \\ 
\Lambda^* (\theta) & \text{if} \ \theta \in [\theta_k, \e [X]),
\end{array}
\right. 
\end{equation}
which ends the proof of Theorem \ref{thm PGD}.

\subsection{Proof of Proposition \ref{prop relation G Q psi BPRE} for the Galton-Watson case}
\label{subsec preuve prop GW}
In this section we assume that $(Z_n)$ is a Galton-Watson process and prove \eqref{eq egalite intro}, which is a particular but simpler case of Proposition \ref{prop relation G Q psi BPRE}.
\begin{proof} 
First note that for the Galton-Watson case, we have
\begin{equation}
\gamma m^r_1 =1.
\end{equation}
For convenience, we shall write $r=r_1$.
Using the additive property of integration and the change of variable $u=tm^k$ for $k \geq 0$, together with Fubini's theorem and the fact that $\gamma m^r =1$, we obtain
\begin{eqnarray}
\label{eq rel1}
\frac{1}{\gamma}  \int_{1}^{\infty} \bar{G}( \phi (u) )  u^{r-1} du &=& \frac{1}{\gamma} \sum_{k=0}^{\infty} \int_{m^k}^{m^{k+1}} \bar{G}( \phi (u) )  u^{r-1} du \notag \\
&=& \frac{1}{\gamma} \sum_{k=0}^{\infty} \int_{1}^{m} \bar{G}( \phi (t m^k) ) (m^r)^k t^{r-1} dt \notag \\
&=& \frac{1}{\gamma}  \int_{1}^{m} \sum_{k=0}^{\infty} \gamma^{-k} \bar{G}( \phi (t m^k) )  t^{r-1} dt.
\end{eqnarray}
Since $\bar G (t) = G(t) - \gamma t$ and $G (\phi (t) ) = \phi (t m )$, we obtain, for any $k \geq 0$,
\begin{eqnarray}
\label{eq telescoping}
\gamma^{-k} \bar G (\phi (t m^k) ) &=& \gamma^{-k} G (\phi (t m^k) ) - \gamma^{-k} \gamma \phi ( t m^k) \notag\\
                       &=& \gamma^{-k} G^{\circ k+1}(\phi (t) ) - \gamma^{k-1} G^{\circ k} ( \phi (t) ) .
\end{eqnarray}
By \eqref{eq telescoping}, using a telescoping argument and the fact that $ \lim_{k \to \infty} \gamma^{-k} G^{\circ k} (t)= Q (t)$, we get 
\begin{eqnarray}
\label{001}
\sum_{k=0}^{\infty} \gamma^{-k} \bar{G}( \phi (t m^k) ) &=& \gamma Q(\phi (t)) - \gamma \phi (t).
\end{eqnarray}
Therefore,  coming back to \eqref{eq rel1} and using \eqref{001}, we have
\begin{equation}
\label{eq relation 1}
\frac{1}{\gamma}  \int_{1}^{\infty} \bar{G}( \phi (u) )  u^{r-1} du = \int_{1}^{m} Q ( \phi (u) )  u^{r-1} du - \int_{1}^{m}\phi (u) u^{r-1} du.
\end{equation}

Moreover, using the change of variable $u=t/m$ and the relations $G( \phi (t/m) )  = \phi (t )$  and $\gamma m^r=1$, we get
\begin{eqnarray}
\label{eq rel2}
\frac{1}{\gamma}  \int_{0}^{1} G( \phi (u) )  u^{r-1} du = \frac{m^{-r}}{\gamma}  \int_{0}^{m} \phi (t)  t^{r-1} dt = \int_{0}^{m} \phi (t)  t^{r-1} dt. \notag
\end{eqnarray}
Therefore, since $\bar G (u) = G(u) - \gamma u$, we obtain
\begin{eqnarray}
\label{eq relation 2}
\frac{1}{\gamma}    \int_{0}^{1} \bar G( \phi (u) )  u^{r-1} du  
&=& \frac{1}{\gamma} \int_{0}^{1} G (\phi (u) ) u^{r-1} du -  \int_{0}^{1} \phi (u) u^{r-1} du  \notag\\
&=& \int_{0}^{m} \phi (u)  u^{r-1} du -  \int_{0}^{1} \phi (u) u^{r-1} du \notag\\
&=& \int_{1}^{m} \phi (u)  u^{r-1} du .
\end{eqnarray}
Finally, using  \eqref{eq relation 1}, \eqref{eq relation 2} and the additive property of integration, we obtain
\begin{equation}
\label{eq egalite}
\frac{1}{\gamma}  \int_{0}^{\infty} \bar{G}( \phi (u) )  u^{r-1} du = \int_{1}^{m} Q ( \phi (u) )  u^{r-1} du,
\end{equation}
which ends the proof of  \eqref{eq egalite intro}.
\end{proof}

\subsection{Proof of Proposition \ref{prop relation G Q psi BPRE}}
\label{subsec preuve prop}
Let $k \geq 1$. For convenience, let $r=r_k$. Using the additive property of integration, the change of variable $u=t\Pi_j^k$ for $j \geq 0$ and Fubini's theorem, we have 
\begin{eqnarray}
\label{eq rel1 BPRE}
\frac{1}{\gamma_k}  \e^{(r)} \left[ \int_{1}^{\infty} \bar{G}_{k,1}( \phi_{\xi} (u) )  u^{r-1} du \right] &=& \frac{1}{\gamma_k}  \e^{(r)} \left[ \sum_{j=0}^{\infty} \int_{\Pi_j}^{\Pi_{j+1}} \bar{G}_{k,1}( \phi_{\xi} (u) )  u^{r-1} du \right] \notag \\
&=& \frac{1}{\gamma_k}  \e^{(r)} \left[ \sum_{j=0}^{\infty} \int_{1}^{m_{j}} \bar{G}_{k,1}( \phi_{\xi} (t \Pi_j) )  \Pi_j^{r} t^{r-1} dt \right] \notag \\
&=& \frac{1}{\gamma_k}   \sum_{j=0}^{\infty} \e^{(r)} \left[ \int_{1}^{m_{j}} \bar{G}_{k,1}( \phi_{\xi} (t \Pi_j) )  \Pi_j^{r} t^{r-1} dt \right] .
\end{eqnarray}
Recall that $\phi_{\xi} (t \Pi_j ) = g_j (\phi_{T^{j} \xi} ( t ) )$, where $g_j (t) = f_0 \circ \ldots \circ f_{j-1} (t)$ is a random function depending on the environment $\xi_0, \ldots, \xi_{j-1}$ and  $T^{j} \xi = (\xi_{j}, \xi_{j+1}, \ldots )$. 
Then, using the change of measure \eqref{changement de mesure}, the independence of the environment sequence $(\xi_i) $  and Fubini's theorem, we see that
\begin{equation}
\label{eq rel2 BPRE}
\e^{(r)} \left[ \int_{1}^{m_{j}} \bar{G}_{k,1}( \phi_{\xi}(t \Pi_j) )  \Pi_j^{r} t^{r-1} dt \right] 
= \e^{(r)} \left[ \int_{1}^{m_{j}} c_r^{-j} \e_{T^{j} \xi} \left[ \bar{G}_{k,1}( g_j (\phi_{T^{j} \xi} ( t ) )  \right]  t^{r-1} dt \right]. 
\end{equation}
Using the fact that $\bar G_{k,1} (t) = G_{k,1} (t) - \gamma_k t^k$ and the relations $\e [ G_{k,n} ( g_j (t) ) ] = G_{k,n+j} (t)$ and $\e [ g_j^k (t) ] = G_{k,j} (t)$, we get
\begin{eqnarray}
\label{eq rel3 BPRE}
&& \e^{(r)} \left[ \int_{1}^{m_{j}} c_r^{-j} \e_{T^j \xi} \left[ \bar{G}_{k,1}( g_j (\phi_{T^{j} \xi} ( t ) )  \right]  t^{r-1} dt \right] \notag \\
&=& \e^{(r)} \left[ \int_{1}^{m_{j}} c_r^{-j} \left[ {G}_{k,j+1}(\phi_{T^{j} \xi} ( t ) ) -{G}_{k,j}(\phi_{T^{j} \xi} ( t ) ) \right]  t^{r-1} dt \right]. 
\end{eqnarray}
Moreover, since the environment sequence $(\xi_0, \xi_1, \ldots )$ is i.i.d., we obtain, for any $j \geq 0$,
\begin{eqnarray}
\label{eq rel4 BPRE}
&&\e^{(r)} \left[ \int_{1}^{m_{j}} c_r^{-j} \left[ {G}_{k,j+1}(\phi_{T^j \xi} ( t ) ) -{G}_{k,j}(\phi_{ T^j \xi} ( t ) ) \right]  t^{r-1} dt \right] \notag \\
&=& \e^{(r)} \left[ \int_{1}^{m_{0}} c_r^{-j} \left[ {G}_{k,j+1}(\phi_{\xi} ( t ) ) -{G}_{k,j}(\phi_{\xi} ( t ) ) \right]  t^{r-1} dt \right]. 
\end{eqnarray}
Therefore, coming back to \eqref{eq rel1 BPRE} and using the fact that $c_r = \gamma_k$ (for $r = r_k$) together with Fubini's theorem, we obtain
\begin{equation}
\label{eq rel5 BPRE}
\frac{1}{\gamma_k}  \e^{(r)} \left[ \int_{1}^{\infty} \bar{G}_{k,1}( \phi_{\xi} (u) )  u^{r-1} du \right] 
= \e^{(r)} \left[ \int_{1}^{m_{0}} \sum_{j=0}^{\infty} \frac{{G}_{k,j+1}(\phi_{\xi} ( t ) )}{\gamma_k^{j+1}} - \frac{ {G}_{k,j}(\phi_{\xi} ( t ) )}{\gamma_k^j}  t^{r-1} dt \right]. 
\end{equation}
Using  a telescoping argument and the assertion that $ \lim_{j \to \infty} \gamma_{k}^j G_{k,j} (t)= Q_k (t)$ for all $t \in [0,1)$, we get 
\begin{eqnarray}
\label{002}
\sum_{j=0}^{\infty} \left[ \frac{{G}_{k,j+1}(\phi_{\xi} ( t ) )}{\gamma_k^{j+1}} - \frac{ {G}_{k,j}(\phi_{\xi} ( t ) )}{\gamma_k^j} \right] &=& Q_k ( \phi_{\xi} (t) ) - G_{k,0} (\phi_{\xi} (t)) \notag \\
&=& Q_k ( \phi_{\xi} (t) ) - \phi_{\xi}^k (t).
\end{eqnarray}
Therefore, by \eqref{eq rel5 BPRE} and \eqref{002}, we have
\begin{eqnarray}
\label{eq rel6 BPRE}
&&\frac{1}{\gamma_k}  \e^{(r)} \left[ \int_{1}^{\infty} \bar{G}_{k,1}( \phi_{\xi} (u) )  u^{r-1} du \right]  \notag\\
&=& \e^{(r)} \left[ \int_{1}^{m_{0}} Q_k (\phi_{\xi} (t) ) t^{r-1} dt \right] - \e^{(r)} \left[ \int_{1}^{m_{0}} \phi_{\xi}^k (t) t^{r-1} dt \right].
\end{eqnarray}

Moreover, using the identity $\phi_{\xi} (u) = f_0 \left( \phi_{T \xi } ( t/ m_0) \right) $, the change of variable $t=u/m_0$, the independence between $\xi_0$ and $T \xi$, the relation $\gamma_k = c_r$ and Fubini's theorem, we get
\begin{eqnarray}
\label{eq rel7 BPRE}
\e^{(r)} \left[ \int_0^{m_0} \phi_{\xi}^k (t) t^{r-1} dt \right] 
&=&  \e^{(r)} \left[ \int_0^{1} f_0^k \left( \phi_{T \xi} (u) \right) m_0^{r} u^{r-1} du \right] \notag \\ 
&=&  \e^{(r)} \left[ \int_0^{1} \e_{T \xi}^{(r)} \left[ f_0^k \left( \phi_{T \xi} (u) \right) m_0^{r} \right] u^{r-1} du \right] \notag \\ 
&=&  \e^{(r)} \left[ \int_0^{1} c_r^{-1} G_{k,1} \left( \phi_{T \xi} (u) \right)  u^{r-1} du \right] \notag \\
&=& \gamma_k^{-1} \e^{(r)} \left[ \int_0^{1} G_{k,1} \left( \phi_{\xi} (u) \right)  u^{r-1} du \right] .
\end{eqnarray}
Therefore, from the identity $\bar{G}_{k,1}(t) = G_{k,1} (t) - \gamma_k t^k$ and \eqref{eq rel7 BPRE}, it follows that
\begin{eqnarray}
\label{eq rel8 BPRE}
\frac{1}{\gamma_k}  \e^{(r)} \left[ \int_{0}^{1} \bar{G}_{k,1}( \phi_{\xi} (u) )  u^{r-1} du \right]  
&=& \frac{1}{\gamma_k}\e^{(r)} \left[ \int_{0}^{1} G_{k,1} (\phi_{\xi} (t) ) t^{r-1} dt \right] - \e^{(r)} \left[ \int_{0}^{1} \phi_{\xi}^k (t) t^{r-1} dt \right] \notag\\
&=& \e^{(r)} \left[ \int_0^{m_0} \phi_{\xi}^k (t) t^{r-1} dt \right]  - \e^{(r)} \left[ \int_{0}^{1} \phi_{\xi}^k (t) t^{r-1} dt \right] \notag\\
&=& \e^{(r)} \left[ \int_{1}^{m_{0}} \phi_{\xi}^k (t) t^{r-1} dt \right]. 
\end{eqnarray}
Finally, using  \eqref{eq rel6 BPRE} and \eqref{eq rel8 BPRE}, we obtain
\begin{equation}
\label{eq egalite}
\frac{1}{\gamma_k} \e^{(r)} \left[  \int_{0}^{\infty} \bar{G}_{k,1} ( \phi_{\xi} (u) )  u^{r-1} du \right] =  \e^{(r)} \left[ \int_{1}^{m_0} Q_k ( \phi_{\xi} (u) )  u^{r-1} du \right],
\end{equation}
which ends the proof of Proposition \ref{prop relation G Q psi BPRE}.

\section{Applications}
\label{sec application}
In this section we present the proofs of Theorems \ref{thm TCL W-Wn}, \ref{thm LDR Rn 1} and \ref{thm LDR Rn 2}  as applications of Theorem \ref{thm harmonic moments Zn}.
In Section \ref{subsec W-Wn} we give the rate of convergence in the central limit theorem for $W-W_n$ where we prove Theorem \ref{thm TCL W-Wn}. In Section  \ref{subsec LD Rn} we deal with the large deviation results for the ratio $R_n=Z_{n+1}/Z_n$, where we prove Theorems \ref{thm LDR Rn 1} and \ref{thm LDR Rn 2}.

\subsection{Central Limit Theorem for $W-W_n$}
\label{subsec W-Wn}

In this section we prove Theorem \ref{thm TCL W-Wn}.

\begin{proof}[Proof of Theorem \ref{thm TCL W-Wn}]
It is well known that $W$ admits the following decomposition:
\[ \Pi_n ( W-W_n) = \sum_{i=1}^{Z_n} \left( W(i) -1 \right),\]
where under $\mathbb{P}_{\xi}$, the random variables $ W(i) \left( i \geq 1 \right) $ are independent of each other and independent of $Z_n$, with common distribution $\mathbb{P}_{\xi} \left( W(i) \in \cdot \right) = \mathbb{P}_{T^n \xi} \left( W  \in \cdot \right) $. Let
\[
\delta_{\infty}^2 (\xi) = \sum_{n=0}^{\infty} \frac{1}{\Pi_n} \left( \frac{m_n (2)}{m_n^2}-1 \right).
\]
The r.v.\ $\delta_{\infty}^2 (\xi)$ is the variance of $W$ under $\mathbb{P}_{\xi}$ (see e.g.\ \cite{hambly}). 
Notice that if $c_0 := \text{essinf} \ \frac{m_0 (2)}{m_0^{2}} >1,$ then $ \delta_{\infty}^2 (\xi)  \geq c_0 -1 > 0$. Therefore, condition $\mathbb{E} Z_1^{2+ \varepsilon} < \infty $ implies that, for all $k \geq 1$, it holds $ \mathbb{E}_k \left| \frac{W-1}{\delta_{\infty}} \right|^{2+ \varepsilon} \leq \frac{C}{c_0-1} \mathbb{E}_k \left| W-1 \right|^{2+ \varepsilon} < \infty$ (see \cite{guivarchliu}).
By the Berry-Esseen theorem, we have for all $x \in \mathbb{R}$,
\[ \left| \mathbb{P}_{\xi} \left( \frac{ \Pi_n ( W-W_n)}{\sqrt{Z_n} \delta_{\infty} (T^n \xi )} \leq x  \right) - \Phi (x) \right| \leq C \mathbb{E}_{T^n \xi } \left| \frac{W-1}{\delta_{\infty}} \right|^{2+ \varepsilon} \mathbb{E}_{\xi} \left[ Z_n^{- \varepsilon /2}  \right]. \]
Taking expectation with $Z_0=k$ and using Theorem \ref{thm harmonic moments Zn}, we get
\begin{eqnarray} 
\label{5.2}
\left| \mathbb{P}_k \left( \frac{ \Pi_n ( W-W_n)}{\sqrt{Z_n} \delta_{\infty} (T^n \xi )} \leq x  \right) - \Phi (x) \right| 
&\leq& C \ \mathbb{E}_{k} \left| \frac{W-1}{\delta_{\infty}} \right|^{2+ \varepsilon} \mathbb{E}_k \left[ Z_n^{- \varepsilon /2} \right] \notag \\
&\leq& C A_{k,n} (- \varepsilon /2).
\end{eqnarray}

\end{proof}

\subsection{Large deviation rate for $R_n$}
\label{subsec LD Rn}

This section is devoted to the proof of Theorems \ref{thm LDR Rn 1} and \ref{thm LDR Rn 2}. 
Recall that $ M_{n,j}$ is defined by \eqref{Mnj-001}, where $N_{n,i} $ are i.i.d. with generating $f_n$, given the environment $\xi$ (see Section \ref{sec main results}).
\begin{proof}[Proof of Theorem \ref{thm LDR Rn 1}]
Since for all $n \in \mathbb{N}$, $M_{n,j} - m_n$ has the same law  as $M_{0,j} - m_0$, and is independent of $Z_n$, we obtain
\begin{eqnarray*}
\mathbb{P}_k (R_n - m_n \in D) &=& \sum_{j \geq k} \mathbb{P} ( M_{n,j} - m_n \in D ) \mathbb{P}_k (Z_n =j) \\                      
                      &\leq& \sum_{j \geq k} \frac{C_1}{j^r} \mathbb{P}_k (Z_n =j) \\
                      &=& C_1 \e_k \left[ Z_n^{-r} \right].
\end{eqnarray*}
The result (\ref{majoration LGN MA}) follows from Theorem \ref{thm harmonic moments Zn}, and (\ref{minoration LGN MA}) follows similarly.
\end{proof}
\begin{proof}[Proof of Theorem \ref{thm LDR Rn 2}]
We stat with a lemma which is a direct consequence of the Marcinkiewicz-Zygmund inequality (see \cite[p. 356]{chow2012probability}).
\begin{lemma}[\cite{liu2001local}, Lemma 1.4] \label{lemma MZ}
Let $(X_i)_{i \geq 1}$ be a sequence of i.i.d.\ centered r.v.'s. Then we have for $p\in (1, \infty)$,
\begin{equation} \label{MZ inequality}
\mathbb{E} \left| \sum_{i=1}^{n} X_i \right|^p \leq
\left\{
\begin{array}{ll}
(B_p)^p \mathbb{E} \left( | X_i |^p \right) n , & \text{if} \  \ 1 < p \leq 2, \\
(B_p)^p \mathbb{E} \left( | X_i |^p \right) n^{p/2} , & \text{if }\ \  p>2,%
\end{array}%
\right.
\end{equation}
where $B_p = 2 \min \left\{ k^{1/2} : k \in \mathbb{N}, k \geq p/2 \right\}$ is a constant depending only on  $p$  (so that $B_p =2 $ if $1<p \leq 2$).
\end{lemma}
We shall prove Theorem \ref{thm LDR Rn 2} in the case when $p \in (1, 2]$. Using the fact that $M_{n,j} - m_n$ has the same law  as $M_{0,j} - m_0$ and is independent of $Z_n$, we obtain after conditioning
\begin{equation}
\label{lien Rn Yn Zn}
\p_k ( | R_n -m_n| > a) = \sum_{j = k}^{\infty} \p ( | M_{0,j}-m_0 | > a ) \p_k ( Z_n = j) .
\end{equation}
Using \eqref{Mnj-001} and the fact that, under $\p_{\xi}$,  the r.v.'s $N_{0,i}-m_0 \ (i=1, \ldots , j)$ are i.i.d.\ centered and with generating function $f_0$, we get from Lemma \ref{lemma MZ} that, for $p \in (1,2]$,
\begin{eqnarray*}
\p_{\xi} ( | M_{0,j}-m_0 | > a ) &\leq & a^{-p} \e_{\xi} \left| M_{0,j} - m_0 \right|^p \\
								 &\leq & \left( \frac{B_p}{a} \right)^p j^{1-p}\e_{\xi} | Z_1 - m_0 |^p .
\end{eqnarray*}
Taking expectation, we obtain
\begin{eqnarray*}
\p( | M_{n,j}-m_n | > a ) &\leq & \left( \frac{B_p}{a} \right)^p j^{1-p} \e | Z_{1} - m_0 |^p .
\end{eqnarray*}
Therefore, coming back to \eqref{lien Rn Yn Zn} and applying Theorem \ref{thm harmonic moments Zn}, we get 
\begin{eqnarray*}
\p_k ( | R_n -m_n| > a) &\leq & \left( \frac{B_p}{a} \right)^p \e | Z_{1} - m_0 |^p \sum_{j = k}^{\infty} j^{1-p} \p_k ( Z_n = j) \\
&=& \left( \frac{B_p}{a} \right)^p \e | Z_{1} - m_0 |^p \e_k \left[ Z_n^{1-p} \right] \\
&=& C_p a^{-p} A_{k,n} (p-1),
\end{eqnarray*}
with $C_p = B_p \e | Z_{1} - m_0 |^p$. This  ends the proof of Theorem \ref{thm LDR Rn 2} in the case when $p \in (1,2]$. The proof in the case $p >2$ is obtained in the same way.
\end{proof}

\bibliographystyle{plain}
\bibliography{biblio}

\begin{thebibliography}{10}

\bibitem{afanasyev2012limit}
V.~I. Afanasyev, C.~B{\"o}inghoff, G.~Kersting, and V.~A. Vatutin.
\newblock Limit theorems for weakly subcritical branching processes in random
  environment.
\newblock {\em J. Theoret. Probab.}, 25(3):703--732, 2012.

\bibitem{afanasyev2014conditional}
V.~I. Afanasyev, C.~B{\"o}inghoff, G.~Kersting, and V.~A. Vatutin.
\newblock Conditional limit theorems for intermediately subcritical branching
  processes in random environment.
\newblock {\em Ann. Inst. Henri Poincar\'e Probab. Stat.}, 50(2):602--627,
  2014.

\bibitem{athreya1971branching2}
K.~B. Athreya and S.~Karlin.
\newblock Branching processes with random environments: {II}: Limit theorems.
\newblock {\em Ann. Math. Stat.}, 42(6):1843--1858, 1971.

\bibitem{athreya1971branching}
K.~B. Athreya and S.~Karlin.
\newblock On branching processes with random environments: I: Extinction
  probabilities.
\newblock {\em Ann. Math. Stat.}, 42(5):1499--1520, 1971.

\bibitem{athreya}
K.~B. Athreya and P.~E. Ney.
\newblock {\em Branching processes}, volume~28.
\newblock Springer-Verlag Berlin, 1972.

\bibitem{bansaye2009large}
V.~Bansaye and J.~Berestycki.
\newblock Large deviations for branching processes in random environment.
\newblock {\em Markov Process. Related Fields}, 15(4):493--524, 2009.

\bibitem{bansaye2011upper}
V.~Bansaye and C.~B{\"o}inghoff.
\newblock Upper large deviations for branching processes in random environment
  with heavy tails.
\newblock {\em Electron. J. Probab.}, 16(69):1900--1933, 2011.

\bibitem{bansaye2013lower}
V.~Bansaye and C.~B{\"o}inghoff.
\newblock Lower large deviations for supercritical branching processes in
  random environment.
\newblock {\em Proc. Steklov Inst. Math.}, 282(1):15--34, 2013.

\bibitem{bansaye2014small}
V.~Bansaye and C.~B{\"o}inghoff.
\newblock Small positive values for supercritical branching processes in random
  environment.
\newblock {\em Ann. Inst. Henri Poincar\'e Probab. Stat.}, 50(3):770--805,
  2014.

\bibitem{boinghoff2014limit}
C.~B{\"o}inghoff.
\newblock Limit theorems for strongly and intermediately supercritical
  branching processes in random environment with linear fractional offspring
  distributions.
\newblock {\em Stoch. Process. Appl.}, 124(11):3553--3577, 2014.

\bibitem{boinghoff2010upper}
C.~B{\"o}inghoff and G.~Kersting.
\newblock Upper large deviations of branching processes in a random environment
  - offspring distributions with geometrically bounded tails.
\newblock {\em Stoch. Process. Appl.}, 120(10):2064--2077, 2010.

\bibitem{chow2012probability}
Y.~S. Chow and H.~Teicher.
\newblock {\em Probability theory: independence, interchangeability,
  martingales}.
\newblock Springer Science \& Business Media, 2012.

\bibitem{dembo1998large}
A.~Dembo and O.~Zeitouni.
\newblock {\em Large deviations techniques and applications}, volume~2.
\newblock Springer, 1998.

\bibitem{glm2016asymptotic}
I.~Grama, Q.~Liu, and E.~Miqueu.
\newblock Asymptotic of the distribution and harmonic moments for a
  supercritical branching process in a random environment.
\newblock {\em arXiv preprint arXiv:1606.04228}, 2016.

\bibitem{guivarchliu}
Y.~Guivarc'h and Q.~Liu.
\newblock Propri{\'e}t{\'e}s asymptotiques des processus de branchement en
  environnement al{\'e}atoire.
\newblock {\em Comptes Rendus de l'Acad{\'e}mie des Sciences-Series
  I-Mathematics}, 332(4):339--344, 2001.

\bibitem{hambly}
B.~Hambly.
\newblock On the limiting distribution of a supercritical branching process in
  a random environment.
\newblock {\em J. Appl. Probab.}, 29(3):499--518, 1992.

\bibitem{liu}
C.~Huang and Q.~Liu.
\newblock Moments, moderate and large deviations for a branching process in a
  random environment.
\newblock {\em Stoch. Process. Appl.}, 122(2):522--545, 2012.

\bibitem{huang_convergence_2014}
C.~Huang and Q.~Liu.
\newblock Convergence in ${L_p}$ and its exponential rate for a branching
  process in a random environment.
\newblock {\em Electron. J. Probab.}, 19(104):1--22, 2014.

\bibitem{kozlov2006large}
M.~V. Kozlov.
\newblock On large deviations of branching processes in a random environment:
  geometric distribution of descendants.
\newblock {\em Discrete Math. Appl.}, 16(2):155--174, 2006.

\bibitem{kozlov2010large}
M.~V. Kozlov.
\newblock On large deviations of strictly subcritical branching processes in a
  random environment with geometric distribution of progeny.
\newblock {\em Theory Probab. Appl.}, 54(3):424--446, 2010.

\bibitem{liu2001local}
Q.~Liu.
\newblock Local dimensions of the branching measure on a {G}alton--{W}atson
  tree.
\newblock {\em Ann. Inst. Henri Poincar\'e Probab. Stat.}, 37(2):195--222,
  2001.

\bibitem{Nakashima2013lower}
M.~Nakashima.
\newblock Lower deviations of branching processes in random environment with
  geometrical offspring distributions.
\newblock {\em Stoch. Process. Appl.}, 123(9):3560--3587, 2013.

\bibitem{ney2003harmonic}
P.~E. Ney and A.~N. Vidyashankar.
\newblock Harmonic moments and large deviation rates for supercritical
  branching processes.
\newblock {\em Ann. Appl. Probab.}, pages 475--489, 2003.

\bibitem{smith}
W.~L. Smith and W.~E. Wilkinson.
\newblock On branching processes in random environments.
\newblock {\em Ann. Math. Stat.}, 40(3):814--827, 1969.

\bibitem{tanny1988necessary}
D.~Tanny.
\newblock A necessary and sufficient condition for a branching process in a
  random environment to grow like the product of its means.
\newblock {\em Stoch. Process. Appl.}, 28(1):123--139, 1988.

\bibitem{Va2010}
V.~A. Vatutin.
\newblock A refinement of limit theorems for the critical branching processes
  in random environment.
\newblock In {\em Workshop on Branching Processes and their Applications. Lect.
  Notes Stat. Proc.}, volume 197, pages 3--19. Springer, Berlin, 2010.

\bibitem{VaZheng2012}
V.~A. Vatutin and X.~Zheng.
\newblock Subcritical branching processes in a random environment without the
  {C}ramer condition.
\newblock {\em Stoch. Process. Appl.}, 122(7):2594--2609, 2012.

\end{thebibliography}

\nocite{smith}

\nocite{athreya}

\nocite{athreya1971branching}

\nocite{athreya1971branching2}

\nocite{liu}

\nocite{huang_convergence_2014}

\nocite{bansaye2009large}

\nocite{boinghoff2010upper}

\nocite{bansaye2011upper}

\nocite{bansaye2012lower}

\nocite{kozlov2010large}

\nocite{kozlov2006large}

\nocite{bansaye2014small}

\end{document}